\documentclass[letterpaper,10pt,reqno]{amsart}
\usepackage{indentfirst} 
\usepackage{amssymb}
\usepackage{mathtools} 
\usepackage{mathabx} 
\usepackage{amsthm}
\usepackage{thmtools}
\usepackage{enumitem} 
\usepackage[colorlinks=true]{hyperref} 
\usepackage[usenames,dvipsnames]{xcolor} 
\usepackage{ifthen}
\usepackage{tikz}
\usetikzlibrary{decorations.pathreplacing}
\usepackage{tikz-cd}
  
\makeatletter

\def\MRbibitem{\@ifnextchar[\my@lbibitem\my@bibitem}

\def\mybiblabel#1#2{\@biblabel{{\hyperref{http://www.ams.org/mathscinet-getitem?mr=#1}{}{}{#2}}}}

\def\myhyperanchor#1{\Hy@raisedlink{\hyper@anchorstart{cite.#1}\hyper@anchorend}}

\def\my@lbibitem[#1]#2#3#4\par{%
  \item[\mybiblabel{#2}{#1}\myhyperanchor{#3}\hfill]#4%
  \@ifundefined{ifbackrefparscan}{}{\BR@backref{#3}}%
  \if@filesw{\let\protect\noexpand\immediate
    \write\@auxout{\string\bibcite{#3}{#1}}}\fi\ignorespaces%
}

\def\my@bibitem#1#2#3\par{%
  \refstepcounter\@listctr
  \item[\mybiblabel{#1}{\the\value\@listctr}\myhyperanchor{#2}\hfill]#3%
  \@ifundefined{ifbackrefparscan}{}{\BR@backref{#2}}%
  \if@filesw\immediate\write\@auxout
    {\string\bibcite{#2}{\the\value\@listctr}}\fi\ignorespaces%
}

\makeatother

\DeclareFontFamily{U} {MnSymbolA}{}
\DeclareFontShape{U}{MnSymbolA}{m}{n}{
   <-6> MnSymbolA5
   <6-7> MnSymbolA6
   <7-8> MnSymbolA7
   <8-9> MnSymbolA8
   <9-10> MnSymbolA9
   <10-12> MnSymbolA10
   <12-> MnSymbolA12}{}
\DeclareFontShape{U}{MnSymbolA}{b}{n}{
   <-6> MnSymbolA-Bold5
   <6-7> MnSymbolA-Bold6
   <7-8> MnSymbolA-Bold7
   <8-9> MnSymbolA-Bold8
   <9-10> MnSymbolA-Bold9
   <10-12> MnSymbolA-Bold10
   <12-> MnSymbolA-Bold12}{}
\DeclareSymbolFont{MnSyA} {U} {MnSymbolA}{m}{n}
 \DeclareFontFamily{U} {MnSymbolC}{}
\DeclareFontShape{U}{MnSymbolC}{m}{n}{
  <-6> MnSymbolC5
  <6-7> MnSymbolC6
  <7-8> MnSymbolC7
  <8-9> MnSymbolC8
  <9-10> MnSymbolC9
  <10-12> MnSymbolC10
  <12-> MnSymbolC12}{}
\DeclareFontShape{U}{MnSymbolC}{b}{n}{
  <-6> MnSymbolC-Bold5
  <6-7> MnSymbolC-Bold6
  <7-8> MnSymbolC-Bold7
  <8-9> MnSymbolC-Bold8
  <9-10> MnSymbolC-Bold9
  <10-12> MnSymbolC-Bold10
  <12-> MnSymbolC-Bold12}{}
\DeclareSymbolFont{MnSyC} {U} {MnSymbolC}{m}{n}

\DeclareMathSymbol{\top}{\mathord}{MnSyA}{219} 
\DeclareMathSymbol{\plus}{\mathord}{MnSyC}{20} 

\declaretheorem[numberwithin=section]{theorem}
\declaretheorem[sibling=theorem]{lemma}
\declaretheorem[sibling=theorem]{corollary}
\declaretheorem[sibling=theorem]{proposition}
\declaretheorem[sibling=theorem,style=definition]{definition}

\declaretheorem[sibling=theorem,style=remark]{remark}

\newtheorem{claim}{Claim}

\hypersetup{bookmarksdepth = 3} 
\numberwithin{equation}{section}     

\setlist[enumerate,1]{label={\upshape(\alph*)},ref=\alph*}
\setlist[enumerate,2]{label={\upshape(\arabic*)},ref=\arabic*}

\newcommand{\M}{\mathcal{M}}
\newcommand{\R}{\mathbb{R}}
\newcommand{\Z}{\mathbb{Z}}
\newcommand{\N}{\mathbb{N}}
\newcommand{\E}{\mathbb{E}}
\renewcommand{\P}{\mathbb{P}}

\newcommand{\cL}{\mathcal{L}}

\newcommand{\cR}{\mathcal{R}}

\def\eps{\varepsilon}
\def\phi{\varphi}
\def\R{{\mathbb R}}

\def\N{{\mathbb N}}
\def\Z{{\mathbb Z}}

\def\E{{\mathcal E}}

\def\P{{\mathcal P}}
\def\Q{{\mathcal Q}}
\def\F{{\mathcal F}}
\def\D{{\mathcal D}}
\def\M{{\mathcal M}}

\def\T{{\mathcal T}}

\def\supp{\mbox{\rm supp}}
\def\diam{\mbox{\rm diam} }

\def\le{\leqslant}
\def\ge{\geqslant}

\def\e{\epsilon}

\def\F{\mathcal{F}}
\def\M{\mathcal{M}}

\def\G{\mathcal{G}}

\newcommand{\vertiii}[1]{{\left\vert\kern-0.25ex\left\vert\kern-0.25ex\left\vert #1 
    \right\vert\kern-0.25ex\right\vert\kern-0.25ex\right\vert}}
\newcommand{\invertiii}[1]{{\vert\kern-0.25ex\vert\kern-0.25ex\vert #1 
    \vert\kern-0.25ex\vert\kern-0.25ex\vert}}

\begin{document}

\title{Escape of entropy for countable Markov shifts}



\begin{thanks}  {GI was partially supported by CONICYT PIA ACT172001 and by Proyecto Fondecyt 1190194.  MT and AV thank UC for their hospitality. AV thanks Paul Apisa for interesting discussions about the set of points that escape on average. Finally, we thank the referee for useful comments and suggestions.}
 \end{thanks}

\author[G.~Iommi]{Godofredo Iommi} \address{Facultad de Matem\'aticas,
Pontificia Universidad Cat\'olica de Chile (PUC), Avenida Vicu\~na Mackenna 4860, Santiago, Chile}
 \email{\href{mailto:giommi@mat.puc.cl}{giommi@mat.puc.cl}} 
\urladdr{\href{http://www.mat.uc.cl/~giommi}{www.mat.uc.cl/$\sim$giommi}}

\author[M.~Todd]{Mike Todd}
\address{Mathematical Institute,
University of St Andrews,
North Haugh,
St Andrews,
KY16 9SS,
Scotland} 
\email{\href{m.todd@st-andrews.ac.uk}{m.todd@st-andrews.ac.uk}}
\urladdr{\href{http://www.mcs.st-and.ac.uk/~miket/}{www.mcs.st-and.ac.uk/$\sim$miket}}

 \author[A.~Velozo]{Anibal Velozo}  \address{Department of Mathematics, Yale University, New Haven, CT 06511, USA.}
\email{\href{anibal.velozo@gmail.com}{anibal.velozo@yale.edu}}
\urladdr{\href{https://gauss.math.yale.edu/~av578/}{https://gauss.math.yale.edu/~av578/}}

\begin{abstract}
In this paper we study ergodic theory of countable Markov shifts. These are dynamical systems defined over non-compact spaces. Our main result relates the escape of mass, the measure theoretic entropy, and the entropy at infinity of the system. This relation has several consequences. For example we obtain that  the entropy map is upper semi-continuous and that the ergodic measures form an entropy dense subset. Our results also provide new proofs of results describing the existence and stability of the measure of maximal entropy. We relate the entropy at infinity with the Hausdorff dimension of the set of recurrent points that escape on average. Of independent interest, we prove a version of  Katok's entropy formula in this non-compact setting.  
\end{abstract}

\maketitle

\section{Introduction}
\label{sec:intro}

Many problems in ergodic theory and dynamical systems involve properties of limits of sequences of invariant probability measures.  If the phase space is compact then the space of invariant probability measures is also compact in the the weak$^*$ topology, which is partly a consequence of convergence in this topology preserving mass.  However, when the phase space is non-compact, the space of invariant probability measures might also be non-compact, and thus mass, as well as other quantities of interest, may escape in the limit.  In this paper, we are principally interested in how the entropy of sequences of measures behaves in this setting.

More specifically, we consider countable Markov shifts (CMS) $(\Sigma,\sigma)$, which in general are not even locally compact.  We discuss the difficulties with the various classical topologies in this context in the next section, where we also give details of the space of invariant sub-probability measures endowed with the so-called cylinder topology, introduced in \cite{iv}. This topology generalises  the  vague topology to a non-locally compact setting (see  Section~\ref{cyl}). If $(\mu_n)_n$ is a sequence of $\sigma$-invariant probability measures that converges in the cylinder topology to the measure $\mu$ then
the total mass $|\mu|:=\mu(\Sigma) \in [0,1]$. In particular, this topology captures the 
escape of mass. Moreover, $\mu$ is an invariant measure and the normalisation $\mu/|\mu|$ is an invariant probability measure (whenever $
\mu$ is not the zero measure).  Denote by $h_{\nu}(\sigma)$ the entropy of the 
invariant probability measure $\nu$ (see Section~\ref{sec:em} for details), and by $\delta_\infty$ the \emph{topological entropy at infinity} of the system (see Definition \ref{def:ent_inf}). Our first main result answers one of the classical 
questions about sequences of measures: how does entropy change in the limit?

\begin{theorem}\label{thm:main} Let $(\Sigma,\sigma)$ be a transitive CMS with finite topological entropy.  Let $(\mu_n)_{n}$ be a sequence of $\sigma$-invariant probability measures converging on cylinders to $\mu$. Then 
\begin{equation} \label{eq:main}
\limsup_{n\to \infty} h_{\mu_n}(\sigma)\le |\mu|h_{\mu/|\mu|}(\sigma)+(1-|\mu|)\delta_\infty.
\end{equation}
If the sequence converges on cylinders to the  zero measure then the right hand side is understood as $\delta_\infty$.
\end{theorem}

Since the topological entropy at infinity plays a crucial role in this article, we define it here, leaving details of the background to this to Section~\ref{sec:introh}. The idea is to measure how complicated the dynamics is near infinity. Of course, such a notion only makes sense for dynamical systems defined on non-compact phase spaces. As in the classical entropy theory, we will study two ways of measuring the complexity of the system near infinity, one topological in nature and the other measure theoretic.

\begin{definition} \label{def:ent_inf} 
Let  $(\Sigma, \sigma)$ be a  CMS. Let $M, q \in \N$. For $n\in \N$ let  $z_n(M, q)$ be the number of cylinders of the form $[x_0,\ldots,x_{n+1}]$, where $x_0\le q$, $x_{n+1}\le q$, and 
\begin{equation*}
\# \left\{ i\in\{0,1,\ldots,n+1\}: x_i\le q \right\} \le \frac{n+2}{M}.
\end{equation*}
Define 
\begin{equation*}
\delta_\infty(M,q):=\limsup_{n\to\infty}\frac{1}{n}\log z_n(M, q),
\end{equation*}
and 
\begin{equation*}
\delta_\infty(q):=\liminf_{M\to\infty} \delta_\infty(M,q).
\end{equation*}
The \emph{topological entropy at infinity} of $(\Sigma,\sigma)$ is defined by $\delta_\infty:=\liminf_{q\to\infty}\delta_\infty(q)$. 
\end{definition}

The measure theoretic counterpart is given by:

\begin{definition} \label{def:ent_meas_inf} Let  $(\Sigma, \sigma)$ be a finite entropy CMS. The \emph{measure theoretic entropy at infinity} of  $(\Sigma, \sigma)$ is defined by
\begin{equation}
h_\infty :=\sup_{(\mu_n)_n\to 0}\limsup_{n\to\infty}h_{\mu_n}(\sigma), \label{eq:mte}
\end{equation}
where $(\mu_n)_n\to 0$ means that the sequence $(\mu_n)_n$ converges  on cylinders to the zero measure. 
\end{definition}

Other authors have considered related concepts. Most notably, Buzzi  \cite[Definition 1.13]{b} proposed a notion of entropy at infinity for CMS. His definition  is given in terms of the graph $G$ which defines the CMS $(\Sigma, \sigma)$: \begin{equation*}
b_{\infty}:=\inf_{F} \inf_{\lambda >0} \sup \left\{h_{\mu}(\sigma) : \mu([F]) < \lambda 			\right\},
\end{equation*} 
where $F$ ranges over the finite sub-graphs of $G$ and $[F]:= \left\{ x \in \Sigma : x_0 \in \mathcal{A}_F \right\}$, where $ \mathcal{A}_F$ denotes the symbols appearing as vertex of $F$.  It turns out that Buzzi's notion coincides with ours. Indeed,  our next result states that all these three notions coincide.

\begin{theorem} \label{thm:vpinf} Let $(\Sigma,\sigma)$ be a CMS of finite topological entropy. Then
\begin{equation*}
 \delta_{\infty}= h_\infty  =b_\infty.
\end{equation*}
\end{theorem}

The equality $ \delta_{\infty}= h_\infty$ can be understood as a variational principle at infinity.

Einsiedler, Lindenstrauss, Michel and Venkatesh \cite[Lemma 4.4]{elmv} were the first to obtain an inequality similar to \eqref{eq:main}. It appeared in their ergodic theoretic proof of Duke's theorem on equidistribution of closed geodesics on the modular surface. After that, similar results in the context of homogeneous dynamics were obtained in \cite[Theorem 1.2]{ek} and \cite[Theorem A]{ekp}. For different classes of geodesic flows defined on non-compact manifolds of negative sectional curvature related results were obtained in \cite[Theorem 1.2]{irv}  and \cite[Theorem 1.1]{rv}. In this context the most general result was obtained in \cite[Theorems 1.4 and 1.6]{ve} where an inequality like  \eqref{eq:main}
was proved for the geodesic flow defined on an arbitrary complete Riemannian manifolds with pinched negative sectional curvature. The manifolds studied are locally compact, thus the topology considered in the space of invariant measures is the vague topology. A more interesting and subtle point is the quantity playing the role of the entropy at infinity. Due to the geometric nature of the examples studied, the entropy at infinity is related to the critical exponent of the Poincar\'e series associated to the non-compact parts of the space (in the geometrically finite case this reduces to the critical exponent of the parabolic subgroups of the fundamental group). Let us mention that the topological entropy at infinity of the geodesic flow was also studied  by Schapira and Tapie \cite{st} in their work  about the rate of change of the topological entropy under perturbations of the metric.

 A major difference with previous works  is that in the context of CMS the behaviour of  the orbits approaching infinity can be very complicated and that  we do not assume the phase space to be locally compact. These are major difficulties that have to be overcome making the analysis more technical.  As a general principle we follow the method employed in \cite{ve} with appropriate modifications. Loosely speaking the entropy at infinity of the geodesic flow counts geodesics that start and end at a given base point, but do not return near this point for intermediate times.  In our setup the entropy at infinity counts orbits that might return near a base point many times, but the number of returns become negligible on average, which can occur due to the lack of local compactness. 

There are several interesting consequences of Theorem \ref{thm:main}, some of them are discussed in Section \ref{sec:app}. For example, in Theorem \ref{semicont} it is proved that the entropy map is upper semi-continuous for every transitive finite entropy CMS. The continuity properties of the entropy map have been studied for a long time. Major results in the area are that for expansive systems defined on compact metric spaces the entropy map is upper semi-continuous \cite[Theorem 8.2]{wa}. Another fundamental result is that if $f$ is a $C^{\infty}$ diffeomorphism defined on a smooth compact manifold then again the entropy map is upper semi-continuous \cite[Theorem 4.1]{n}. As explained in Remark \ref{rem:nousc}, for infinite entropy CMS the entropy map is not upper semi-continuous. In a recent article we proved \cite[Corollary 1.2]{itv} that if $(\Sigma, \sigma)$ is a finite entropy transitive CMS then the entropy map is upper semi-continuous when restricted to ergodic measures. A complete solution to the problem can be obtained as a consequence of Theorem \ref{thm:main}.  In Section  \ref{sec:app} we also prove that the set of ergodic measures is  `entropy dense' in the space of invariant probability measures. This result not only provides a fine description of the structure of the space of invariant probability measures but also provides an important  tool to study Large Deviations in this setting. 

There is a classification of transitive CMS in terms of their recurrence properties: they can be transient, null recurrent or positive recurrent (see Definition \ref{def:clas} for $\phi=0$). Positive recurrent CMS are precisely those with a measure of maximal entropy. A particularly important role is played by strongly positive recurrent CMS (SPR); which are a sub-class of positive recurrent Markov shifts. The dynamical properties of this class of systems are similar to that of sub-shifts of finite type.  Buzzi gave a characterisation of SPR shifts using $b_\infty$ in \cite[Proposition 6.1]{b}, based on the work of  Gurevich-Zargaryan, Gurevich-Savchenko and Ruette.  We note in Proposition \ref{prechar} that we can now restate this result as saying that $(\Sigma,\sigma)$ is SPR if and only if $\delta_\infty<h_{top}(\sigma)$, where $h_{top}(\sigma)$ is the Gurevich entropy of $(\Sigma,\sigma)$ (for precise definitions see Section \ref{sec:tf}).  In Section \ref{sec:mme} we use Theorem \ref{thm:main} to obtain stability properties of the measure of maximal entropy for SPR CMS (recovering results from \cite{gs}). Similar arguments are used to prove the existence of equilibrium states for potentials in $C_0(\Sigma)$, the space of test functions for the cylinder topology (see Section \ref{sec:eqst}).  To the author's knowledge, this is the first result on the existence of equilibrium states for CMS that goes beyond regular potentials (e.g. with summable variations or the Walters property). Finally, in Theorem \ref{thm:em} we prove that for SPR systems it is possible to bound the amount of mass that escapes the system in terms of the entropy of the measures. Sequences of measures with large entropy can not  lose  much mass. 

The entropy at infinity has yet another important appearance in dynamics: it is related to the Hausdorff dimension of the set of points that escape on average (see  \cite{aaekmu, ekp, kklm, kp}). These are points for which the frequency of visits to every cylinder equals to zero. In particular, no invariant measure is supported on that set. This notion has been studied recently in contexts of homogeneous and Teichm\"uller dynamics. The motivation comes from work of Dani \cite{da} in the mid 1980s who proved that singular matrices are  in  one-to-one correspondence with certain divergent orbits of one parameter diagonal groups on the space of lattices.  In Theorem \ref{thm:onave} we prove that the Hausdorff dimension of the set of recurrent points that escape on average is bounded above by $\delta_{\infty} / \log 2$, where the factor $\log 2$ comes from the metric in the symbolic space.

While our interest in this paper lies in the realm of Markov shifts, to provide context we mention some applications of this theory. Symbolic methods have been used to describe dynamical properties of a variety of systems since the $1898$ work of Hadamard on closed geodesics on surfaces of negative curvature, at the latest. Compact Markov shifts have been used to study uniformly hyperbolic dynamical systems defined on compact spaces, see for example the work of Bowen in  \cite{bo2}. Many deep results can be obtained with this coding. Mostly after the work of Sarig \cite{sa4}, countable Markov partitions have been constructed for a wide range of dynamical systems. This gives a semiconjugacy between a relevant part of the dynamics, albeit not all of it,  and a CMS.  Examples of systems for which Markov partitions have been constructed include positive entropy diffeomorphisms defined on compact manifolds  \cite{b2,ov, sa4} and Sinai and Bunimovich billiards \cite{lm}. Remarkable results have been proved making use of these codings, for example in  \cite[Main Theorem]{bcs} it is shown that a positive entropy $C^{\infty}$ diffeomorphism of a closed surface admits at most finitely many ergodic measures of maximal entropy. Results in this paper apply to all the symbolic codings mentioned above. However, due to topologies possibly not being preserved by the coding, it is not clear that the results pass to the original systems.

In 1980 Katok \cite[Theorem 1.1]{ka}  established a formula for the entropy of an invariant probability measure in analogy to the definition of topological entropy of a dynamical system \cite{bo,d}. This formula is now known as \emph{Katok's entropy formula}. An important assumption in  \cite[Theorem 1.1]{ka} is the compactness of the phase space. In Section \ref{ent} we prove that Katok's entropy formula holds in the non-compact  setting of CMS. We require this formula in the proof of Theorem \ref{thm:main}, but it is also of independent interest.

\section{Preliminaries}

\subsection{Basic definitions for CMS} \label{sec:defcms}

Let $M$ be a $\N\times \N$ matrix  with entries $0$ or $1$. The symbolic space associated to $M$ with alphabet $\N$ is defined by
 \begin{equation*}
 \Sigma:=\left\{ (x_0, x_1, \dots) \in \N^{\N_0}: M(x_i, x_{i+1})=1 \text{ for every } i \in \N_0 \right\},
\end{equation*} 
where $\N_0:=\N \cup \{0\}$. We endow $\N$ with the discrete topology and $\N^{\N_0}$ with the product topology. On $\Sigma$ we consider the induced topology given by the natural inclusion $\Sigma\subset \N^{\N_0}$. We stress that, in general, this is a non-compact space.  
The space $\Sigma$ is locally compact if and only if for every $i \in \N$  we have $\sum_{j \in \N} M(i,j ) <\infty$ (see \cite[Observation 7.2.3]{ki}).

The \emph{shift map} $\sigma:\Sigma \to \Sigma$ is defined by $(\sigma(x))_i=x_{i+1}$, where $x=(x_0, x_1, \dots ) \in \Sigma$. Note that $\sigma$  is a continuous map. The pair $(\Sigma,\sigma)$ is called  a one sided \emph{countable Markov shift} (CMS).  The matrix $M$ can be identified with a directed graph $G$ with no multiple edges (but allowing edges connecting a vertex to itself). 

An \emph{admissible word}  of length $N$ is a string ${\bf w}  =a_0a_1\ldots a_{N-1}$ of letters in the alphabet such that $M(a_i,a_{i+1})=1$, for every $i\in\{0,\ldots,N-2\}$. We use bold letters to denote admissible words. The length of an admissible word ${\bf w}$ is $\ell({\bf w})$. 

A \emph{cylinder} of length $N$ is the set
\begin{equation*}
[a_0,\ldots,a_{N-1}]:= \left\{ x=(x_0,x_1,\ldots)\in \Sigma :  x_i=a_i  \text{ for } 0 \le i \le N-1 \right\}.
\end{equation*} 
If $a_0\ldots a_{N-1}$ is an admissible word then $[a_0,\ldots,a_{N-1}] \neq \emptyset$. We use the notation $C_n(x)$ to denote the cylinder of length $n$ containing $x$.  Since a cylinder can be identified with an admissible word, we also denote the length of a cylinder $C$  by $\ell(C)$. Note that the topology generated by the cylinder sets coincides with that induced by the product topology. 

The space $\Sigma$ is metrisable. Indeed, let $d:\Sigma \times \Sigma \to \R$ be the function defined by
\begin{equation} \label{metric}
d(x,y):=
\begin{cases}
1 & \text{ if } x_0\ne y_0; \\
2^{-k} & \text{ if  } x_i=y_i \text{ for  } i \in \{0, \dots , k-1\} \text{ and } x_k \neq y_k; \\
0 & \text{ if } x=y.
\end{cases}
\end{equation} 
The function $d$ is a metric and it generates the same topology as that of the cylinders sets. Moreover,  the ball $B(x,2^{-N})$ is precisely $C_N(x)$.  Given $\phi:\Sigma\to \R$, we define 
\begin{equation*}
\text{var}_n(\phi):=\sup \left\{|\phi(x)-\phi(y)|: \forall x,y\in \Sigma \text{ such that }d(x,y)\le 2^{-n} \right\}.
\end{equation*}
A function $\phi:\Sigma\to \R$ is said to have \emph{summable variations} if $\sum_{n\ge 2}\text{var}_n(\phi)<\infty$. A function $\phi$ is called \emph{weakly H\"older} if there exist $\theta \in (0,1)$  and a positive constant $O \in \R$  such that $\text{var}_n(\phi)\le O \theta^n$, for every $n\ge 2$.  A weakly H\"older continuous function is H\"older if and only if it is bounded. The $C^0$-norm of $\phi$ is  $\|\phi\|_0:=\sup_{x\in \Sigma}|\phi(x)|$.  We denote by
\begin{equation*}
S_n\phi(x)=\sum_{k=0}^{n-1}\phi(\sigma^k x),
\end{equation*}
the \emph{Birkhoff sum} of $\phi$ at the point $x$. 

We say that $(\Sigma,\sigma)$ is \emph{topologically transitive} if its associated directed graph $G$ is connected. We say that $(\Sigma,\sigma)$ is \emph{topologically mixing} if for each pair $a,b\in \N$, there exists a number $N(a,b)$ such that for every $n\ge N(a,b)$ there is an admissible word of length $n$ connecting $a$ and $b$.  There is a particular class of CMS that will be of interest to us,

\begin{definition} \label{def:F}  
A CMS $(\Sigma, \sigma)$ is said to satisfy the $\F-$\emph{property} if  for every element of the alphabet
$a$ and natural number $n$, there are only finitely many admissible words of length $n$ starting and ending at 
$a$.
\end{definition}

\begin{remark}
A CMS $(\Sigma, \sigma)$ satisfies the $\F-$\emph{property} if and only if there are only finitely many periodic orbits of length $n$ intersecting $[a]$, for every $n\in \N$ and for every $a$ in the alphabet.  Note that every locally compact CMS satisfies the $\F-$\emph{property}. 
\end{remark}

\begin{remark}
Equivalent definitions and properties as those discussed in this section can be given for \emph{two sided CMS}. In this case the acting group is $\Z$.
It turns out that, in general,  thermodynamic formalism for the two sided CMS can be reduced to the one sided case (see \cite[Section 2.3]{sabook}). 
\end{remark}

\subsection{Topologies in the space of invariant measures} \label{sec:topo}

The space of invariant measures can be endowed with different topologies, some of which can account for the escape of mass phenomenon whereas others can not. In this section we not only fix notation for later use, but we also recall definitions and properties of several topologies in the space of measures.  First note that in this article a measure is always a countably additive non-negative Borel measure defined in the symbolic space $\Sigma$. The mass of a measure $\mu$ is defined as $|\mu|:=\mu(\Sigma)$.

Denote by  $\M(\Sigma,\sigma)$ the space of  $\sigma$-invariant probability measures on $\Sigma$ and by $\M_{\le 1}(\Sigma,\sigma)$ the space of $\sigma$-invariant measures on $\Sigma$ with mass in $[0,1]$. In other words, $\M_{\le 1}(\Sigma,\sigma)$ is the space of $\sigma$-invariant sub-probability measures on $\Sigma$. Note that $\M(\Sigma,\sigma)\subset \M_{\le 1}(\Sigma,\sigma)$.  The set of ergodic probability measures is denoted by $\mathcal{E}(\Sigma, \sigma)$. 

\subsubsection{The weak$^*$ topology} \label{weak*}

Denote by  $C_b(\Sigma)$ the space of real valued bounded continuous function on $\Sigma$. A sequence of measures $(\mu_n)_n$ in  $\M(\Sigma,\sigma)$ converges to a measure $\mu$ in the weak$^*$ topology if for every $f \in C_b(\Sigma)$ we have 
\begin{equation*}
\lim_{n \to \infty} \int f d \mu_n = \int f d \mu.
\end{equation*}
Note that since the constant function equal to one belongs to $C_b(\Sigma)$ the measure $\mu$ is also a probability measure.  A basis for this topology is given by the collection of sets of the form
\begin{align}\label{defbasis}
V(f_1,\ldots,f_k,\mu,\e):= \left\{	 \nu \in  \M(\Sigma,\sigma) : \left|\int f_i d \nu - \int f_i d \mu   \right| < \epsilon, \text{ for } i \in \{1, \dots, k\}		\right\},
\end{align}
where  $\mu \in  \M(\Sigma,\sigma)$,  $(f_i)_{i}$ are elements from  $C_b(\Sigma) $ and $\epsilon >0$. Note that in this notion of convergence we can replace the set of test functions (bounded and continuous) by the space of bounded uniformly continuous functions (see \cite[8.3.1 Remark]{bg}). 
Convergence with respect to the weak$^*$ topology can be characterised as follows, see \cite[Theorem 2.1]{bi}.

\begin{proposition}[Portmanteau Theorem] \label{port}
Let $(\mu_n)_n, \mu$ be probability measures on $\Sigma$. The following statements are equivalent. 
\begin{enumerate}
\item The sequence $(\mu_n)_n$ converges to $\mu$ in the weak$^*$ topology.
\item If $O \subset \Sigma$ is an open set, then $\mu(O) \leq  \liminf_{n \to \infty} \mu_n(O)$.
\item If $C \subset \Sigma$ is a closed set, then $\mu(C) \geq  \limsup_{n \to \infty} \mu_n(C)$.
\item  If $A \subset \Sigma$ has $\mu(\partial A)=0$, where $\partial A$ is the boundary of $A$, then $ \lim_{n \to \infty} \mu_n(A)= \mu(A)$.
\end{enumerate}
\end{proposition}

Note that the space $\M(\Sigma,\sigma)$  is closed in the weak$^*$ topology (\cite[Theorem 6.10]{wa}). If $\Sigma$ is compact then so is  $\M(\Sigma,\sigma)$   with respect to the weak$^*$ topology (see \cite[Theorem 6.10]{wa}). If $\Sigma$ is not compact then, in general (e.g., whenever the $\F$-property holds), $\M(\Sigma,\sigma)$ is not compact with respect to the weak$^*$ topology. 
Finally,  the space  $\M(\Sigma,\sigma)$ is a convex set whose extreme points are ergodic measures  (see \cite[Theorem 6.10]{wa}). 

\subsubsection{The topology of convergence on cylinders}  \label{cyl}
In this section we recall the definition and properties of the topology of convergence on cylinders. This topology was introduced and studied in \cite{iv} as a way to compactify $\M(\Sigma,\sigma)$  under  suitable assumptions on $\Sigma$.

Let  $(C^n)_n$ be an enumeration of the cylinders of $\Sigma$. Given $\mu, \nu\in \M_{\le 1}(\Sigma,\sigma)$ we define $$\rho(\mu,\nu)=\sum_{n=1}^\infty \frac{1}{2^n}|\mu(C^n)-\nu(C^n)|.$$
It follows from the outer regularity of Borel measures on metric spaces that $\rho(\mu,\nu)=0$, if and only if $\mu=\nu$. Moreover, the function $\rho$ defines a metric on $\M_{\le1}(\Sigma,\sigma)$.  The topology induced by this metric is called the \emph{topology of convergence on cylinders}. We say that a sequence $(\mu_n)_n$   in  $\M_{\le 1}(\Sigma,\sigma)$ \emph{converges on cylinders} to $\mu$ if $$\lim_{n\to\infty}\mu_n(C)=\mu(C),$$ for every cylinder $C\subset \Sigma$.  Of course, $(\mu_n)_n$ converges on cylinders to $\mu$ if and only if $(\mu_n)_n$ converges to $\mu$ in the topology of convergence on cylinders. 
In the next lemma we see that in the case that mass does not escape then weak$^*$ and convergence on cylinders coincide. 

\begin{lemma}{\cite[Lemma 3.17]{iv}}\label{restriction}
\label{equivtop} Let $(\Sigma,\sigma)$ be a CMS,  $\mu$ and $(\mu_n)_n$ be invariant probability measures on $\Sigma$. The following assertions are equivalent.
\begin{enumerate}
\item The sequence $(\mu_n)_n$ converges in the weak$^*$ topology to $\mu$.
\item The sequence $(\mu_n)_n$ converges on cylinders to $\mu$. 
\end{enumerate}
\end{lemma}

Let $\Sigma$ be a locally compact space and $(\mu_n)_n, \mu$ in  $\M_{\le1}(\Sigma,\sigma)$. The sequence $(\mu_n)_n$ converges to  $\mu$ in the \emph{vague topology} if $\lim_{n \to \infty} \int f d \mu_n = \int f d \mu,$ for every function $f$ continuous and of compact support (note that the set of test functions can be replaced by the set of continuous functions vanishing at infinity).  If $(\Sigma, \sigma)$ is locally compact then the topology of convergence on cylinders coincides with the vague topology (see \cite[Lemma 3.18]{iv}). It is important to note that if $\Sigma$ is transitive and  not locally compact then the space of continuous functions of compact support is trivial, so the vague topology is of no use and the topology of convergence on cylinders is a suitable generalisation (see \cite[Remark 3.13]{iv}).

If $C$ is a cylinder of length $m$, denote by \begin{equation*}
 C(\ge n):= \left\{ x \in C : \sigma^m(x)\in \bigcup_{k\ge n}[k]  \right\}. 
\end{equation*} 
For a non-empty set $ A\subset \Sigma$ we define 
\begin{equation*}
var^A(f):=\sup  \left\{   \left|f(x)-f(y) \right| : x, y \in A \right\}.
\end{equation*}
We declare $var^A(f)=0$ if $A$ is the empty set.

\begin{definition}\label{C_0} We say that $f$ belongs to $C_0(\Sigma)$ if the following three conditions hold:
\begin{enumerate}
\item $f$ is uniformly continuous. 
\item $\lim_{n\to\infty}\sup_{x\in [n]}|f(x)|=0.$
\item $\lim_{n\to\infty}var^{C(\ge n)}(f)=0,$ for every cylinder $C\subset \Sigma$. 
\end{enumerate}
In this case we say that $f$ \emph{vanishes at infinity}. 
\end{definition}
The set $C_0(\Sigma)$ is  the space of test functions for the cylinder topology (see \cite[Lemma 3.19]{iv}). In other words, $(\mu_n)_n$ is a sequence in  $\M_{\le 1}(\Sigma,\sigma)$ that converges in the cylinder topology to $\mu\in\M_{\le 1}(\Sigma,\sigma)$ if and only if $\lim_{n\to\infty}\int f d\mu_n=\int fd\mu$, for every $f\in C_0(\Sigma)$. Since cylinder topology generalises the vague topology for non-locally compact CMS,  the space $C_0(\Sigma)$ is the natural substitute to the set of continuous functions that vanish at infinity. 

The following result was proved in \cite[Theorem 1.2]{iv}, and is an important ingredient for many of our applications.

\begin{theorem}\label{compact} Let $(\Sigma,\sigma)$ be a transitive CMS with the $\F-$property. Then $\M_{\le 1}(\Sigma,\sigma)$ is a compact metrisable space with respect to the cylinder topology. Moreover, $\M_{\le1}(\Sigma,\sigma)$ is affine homeomorphic to the Poulsen simplex. \end{theorem}
 
We remark that, as shown in \cite[Proposition 4.19]{iv}, Theorem \ref{compact} is sharp in a strong sense: if $(\Sigma, \sigma)$ does  not have the $\F-$property, then $\M_{\le1}(\Sigma,\sigma)$ is not compact. More precisely, there exists a sequence of periodic measures that converges on cylinders to a finitely additive measure which is not countably additive.

\subsection{Entropy of a measure} \label{sec:em}

In this section we recall the definition of  the entropy of an invariant measure $\mu \in\M(\Sigma,\sigma)$ (see \cite[Chapter 4]{wa} for more details). We take the opportunity to fix the notation that will be used in what follows.  A partition $\beta$ of a probability space  $(\Sigma, \mu)$ is a countable (finite or infinite) collection of pairwise disjoint subsets of $\Sigma$ whose union has full measure.
The \emph{entropy} of the partition $\beta$ is defined by
\begin{equation*}
H_\mu(\beta):= - \sum_{P \in \beta} \mu(P) \log \mu(P),
\end{equation*}
where $0 \log 0 :=0$. 
It is possible that $H_\mu(\beta)=\infty$. Given two partitions $\beta$ and $\beta'$ of $\Sigma$ we define the new partition 
\begin{equation*}
\beta \vee \beta':= \left\{P \cap Q : P \in  \beta , Q \in \beta' \right\}.
\end{equation*}
Let  $\beta$ be a partition of $\Sigma$. We define the partition $\sigma^{-1}\beta:= \left\{ \sigma^{-1}P : P \in \beta \right\}$ and for $n \in \N$ we set  
$\beta^n:=\bigvee_{i=0}^{n-1} \sigma^{-i}\beta$.  Since the measure $\mu$ is $\sigma$-invariant, the sequence $H_{\mu}(\beta^n)$ is sub-additive. 
The \emph{entropy of $\mu$ with respect to} $\beta$ is defined by
\begin{equation*}
h_{\mu}(\beta):=  \lim_{n \to\infty} \frac{1}{n} H_{\mu}(\beta^n).
\end{equation*}
Finally, the \emph{entropy} of $\mu$ is defined by
\begin{equation*}
h_{\mu}(\sigma):= \sup \left\{h_{\mu}(\beta) : \beta\text{ a partition with } H_{\mu}(\beta) < \infty		\right\}.		
\end{equation*}

\begin{remark}
Krengel  \cite[Remark p.166]{kr} observes that since the entropy of a finite invariant measure $\mu$ is usually defined as the entropy of the normalised measure $\mu/ |\mu|$, the linearity (in the standard sense) of the entropy map is destroyed. Following Krengel's line of thought, the number $ |\mu|h_{\mu/|\mu|}(\sigma)$  appearing in Theorem \ref{thm:main} can be understood,  as the entropy of the finite measure $\mu$ (see also \cite[Theorem 8.1]{wa} for example).
\end{remark}

\subsection{Thermodynamic formalism for CMS} \label{sec:tf}

Throughout this section we assume that $(\Sigma, \sigma)$ is  topologically transitive and that
$\phi:\Sigma\to \R$ has summable variations. Let $A \subset \Sigma$ and  $1_{A}(x)$   be the characteristic function of the set $A$.
In this setting we define,
 \begin{equation*}
 Z_n(\phi,a):=\sum_{\sigma^n x=x} e^{S_n \phi(x)}1_{[a]}(x),
\end{equation*} 
where $a\in \N$.  The \emph{Gurevich pressure} of $\phi$ is defined by
\begin{equation*}
 P_G(\phi):=\limsup_{n\to \infty} \frac{1}{n}\log Z_n(\phi,a).
\end{equation*} 
This definition was introduced by Sarig  \cite{sa1}, based on the work of Gurevich \cite{gu2}. We remark that the right hand side in the definition of $P_G(\phi)$ is independent of $a\in \N$, and that if $(\Sigma,\sigma)$ is topologically mixing, then the limsup can be replaced by a limit  (see  \cite[Theorem 1]{sa1} and  \cite[Theorem 4.3]{sabook}). 
 This definition of pressure satisfies          the variational principle (see \cite[Theorem 3]{sa1} and \cite[Theorem 2.10]{ijt}) and can be approximated by the pressure of compact invariant subsets \cite[Theorem 2 and Corollary 1]{sa1}. Indeed, 
\begin{eqnarray*} \label{thm:vp}
P_G(\phi) &=& \sup \left\{P_{\text{top}}(\phi |K) : K \subset \Sigma \text{ compact and }  \sigma^{-1}K=K 	\right\} \\
 &=& \sup_{\mu\in \M(\Sigma,\sigma)} \left\{h_\mu(\sigma)+\int\phi d\mu: \int \phi d\mu>-\infty \right\},
\end{eqnarray*}
where $P_{\text{top}}( \cdot)$ is the classical pressure on compact spaces \cite[Chapter 9]{wa}. A measure  $\mu \in \M(\Sigma, \sigma)$ is an \emph{equilibrium state} for $\phi$ if   $\int \phi d\mu>-\infty$ and
\begin{equation*}
 P_G(\phi)=h_\mu(\sigma)+\int \phi d\mu.
\end{equation*} 

The following function will be of importance in this article.
\begin{definition} \label{def:ret}
Let $A \subset \Sigma$. Denote by  $R_A(x):=1_{A}(x)\inf\{n\ge1:\sigma^n x\in A\}$ the first return time map to the set $A$. In the particular case in which the set $A$ is a cylinder $[a]$ we denote $R_{[a]}(x):= R_a(x)$.
 \end{definition}

Sarig \cite[Section 4.2]{sa1} introduced the following: 
 \begin{equation*}
 Z_n^*(\phi,a):=\sum_{\sigma^n(x)=x} e^{S_n\phi(x)}1_{[R_a=n]}(x),
\end{equation*} 
where $[R_a =n]:=\left\{x \in \Sigma : R_a(x)=n	\right\}$. Extending notions of Markov chains, Sarig \cite{sa1} classified potentials according to its recurrence properties. 

\begin{definition}[Classification of potentials] \label{def:clas} Let $(\Sigma , \sigma)$ be a topologically transitive  CMS and $\phi$ a summable variation potential with finite Gurevich pressure. Define $\lambda=\exp \left( P_G(\phi) \right)$  and fix $a\in \N$. 
\begin{enumerate}
\item If $\sum_{n\ge 1}\lambda^{-n}Z_n(\phi,a)$ diverges we say that $\phi$ is \emph{recurrent}.
\item If $\sum_{n\ge 1}\lambda^{-n}Z_n(\phi,a)$ converges we say that $\phi$ is \emph{transient}.
\item If $\phi$ is recurrent and $\sum_{n\ge 1}n\lambda^{-n}Z^*_n(\phi,a)$ converges we say that $\phi$ is \emph{positive recurrent}.
\item If $\phi$ is recurrent but $\sum_{n\ge 1}n\lambda^{-n}Z^*_n(\phi,a)$ diverges we say that $\phi$ is \emph{null recurrent}.
 \end{enumerate}
\end{definition}

Topological transitivity implies that above definitions do not depend on the choice of the symbol $a$.

\begin{remark}
The classification in Definition \ref{def:clas} is invariant under the addition of coboundaries and constants. That is,  if $\psi:\Sigma \to \R$ is of summable variations and $C \in \R$ we have that: the potential $\phi$ is recurrent (resp. transient) if and only if the potential $\phi + \psi - \psi \circ \sigma + C$   is recurrent (resp. transient). Moreover,
the potential $\phi$ is positive recurrent (resp. null recurrent) if and only if the potential $\phi + \psi - \psi \circ \sigma + C$   is positive recurrent (resp. null recurrent).
\end{remark}

The following result describes existence and  uniqueness of equilibrium states. Parts (\ref{sumvar1}) and  (\ref{sumvar2}) follow from Theorems 1.1 and Theorem 1.2 of \cite{bs}, respectively.

\begin{theorem} \label{clas} Let $(\Sigma , \sigma)$ be a topologically transitive CMS and $\phi$ a summable variation potential with finite Gurevich pressure. Then 
\begin{enumerate}
\item \label{sumvar1} There exists at most one equilibrium state for $\phi$. 
\item \label{sumvar2}  If $\phi$ has an   equilibrium state then $\phi$ is positive recurrent.
\end{enumerate}
\end{theorem}

 In this article the potential $\phi=0$ will play a  particularly important role. The \emph{topological entropy} of $(\Sigma,\sigma)$, that we denote by $h_{top}(\sigma)$, is defined as the Gurevich pressure of the potential $\phi=0$, that is  
 $$h_{top}(\sigma):=P_G(0).$$
  We say that $(\Sigma,\sigma)$ is recurrent, transient, null recurrent or positive recurrent according to the corresponding properties of $\phi=0$. If $(\Sigma,\sigma)$ is positive recurrent, then Theorem \ref{clas} implies that $(\Sigma,\sigma)$ admits a unique measure of maximal entropy. This was first  proved by Gurevich \cite{gu1}. 
 
 \begin{remark}
 Note that every finite entropy, transitive CMS satisfies the $\F-$property (see Definition  \ref{def:F}).
  \end{remark}

\subsection{Strongly positive recurrent CMS} Properties of CMS may be significantly different from those of sub-shifts of finite type defined on finite alphabets. 
In this section we describe a special class of CMS with  properties analogous to those of compact sub-shifts. This study is based on  work of Vere-Jones \cite{v1,v2} developed during the 1960s, where he first defined an equivalent class in the setting of stochastic matrices.  Several people have contributed to the understanding of this class, for example,  
Salama \cite{s}, Gurevich and Savchenko \cite{gs},  Sarig \cite{sa3}, Ruette \cite{r}, Boyle, Buzzi and G\'omez \cite{bbg} and  Cyr and Sarig \cite{cs}.  In these works the following quantities, or related ones,  have been defined and studied.

\begin{definition}
Let $(\Sigma, \sigma)$ be topologically transitive CMS and $a \in \N$. Let  
\begin{equation*}
\Delta_\infty ([a]):=\limsup_{n\to \infty} \frac{1}{n}\log Z_n^*(0,a),
\end{equation*}
and 
\begin{equation*}
{\Delta}_\infty:=\inf_{a\in \N} \Delta_\infty([a]).
\end{equation*}
\end{definition}

\begin{remark}
The number $\Delta_\infty ([a])$ can depend on the symbol $a$, see \cite[Remark 2,1]{r}.
\end{remark}

\begin{definition}[Strongly positive recurrent CMS] \label{def:spro} Let $(\Sigma,\sigma)$ be a topologically transitive CMS with finite topological entropy. We say that $(\Sigma,\sigma)$ is \emph{strongly positive recurrent} (SPR) if ${\Delta}_\infty<h_{top}(\sigma)$.
\end{definition}

\begin{remark} \label{rem_conspr} A strongly positive recurrent CMS is positive recurrent. In particular it admits a unique measure of maximal entropy. Moreover, with respect to this measure the system $(\Sigma, \sigma)$ is exponentially recurrent (see \cite[Proposition 2.3]{bbg} for precise statements). The class of strongly positive recurrent CMS was intensively studied by Gurevich and Savchenko in \cite{gs}. Note, however, that in \cite{gs} these are called  \emph{stable-positive recurrent}.  We also remark that there exists CMS that are positive recurrent but not strongly positive recurrent (see \cite[Example 2.9]{r}).  
\end{remark}

 \begin{remark} \label{rem:spr}
 Strongly positive recurrent CMS have the property that the entropy is concentrated inside the system and not near infinity. Indeed, let $(\Sigma, \sigma)$ be a CMS an $G$ its associated graph. Gurevich and Zargaryan \cite{gz} (see also \cite{gs}) showed that a condition equivalent to SPR is the existence of a finite connected subgraph $H \subset G$  such that there are more paths inside than outside $H$ (in term of exponential growth). See \cite[Section 3.1]{r} for precise statements. 
On the other hand, for graphs that are not strongly positive recurrent the entropy is supported by the infinite paths that spend most of the time outside a finite subgraph  (see \cite[Proposition 3.3]{r}).
 \end{remark}

Along the lines of the observations made in Remark \ref{rem:spr}, in the next section (see Proposition \ref{prechar}) we characterise SPR for CMS as those having entropy at infinity strictly smaller than the topological entropy.

Sarig \cite{sa3} generalised the notion of strong positive recurrence to potentials $\phi$.  Using his definition, we recover the topological notion in Definition~\ref{def:spro} for the potential $\phi\equiv 0$, i.e., this potential is strongly positive recurrent if and only if  $(\Sigma,\sigma)$ is SPR (see \cite[Remark 2.11]{r}).
For Sarig's class of potentials the associated thermodynamic formalism enjoys most of the properties of the thermodynamics for H\"older potentials on sub-shifts of finite type. In particular, the transfer operator corresponding to a strongly positive recurrent potential has a spectral gap (see \cite[Theorem 2.1]{cs}). This readily implies that the pressure function is analytic and there exist formulas for its derivatives (\cite[Theorem 3 and 4]{sa3} and  \cite[Theorem 1.1]{cs}), there exists a unique equilibrium state and it has exponential decay of correlations and satisfies the Central Limit Theorem (\cite[Theorem 1.1]{cs}).   Moreover, in the space of potentials  strongly positive recurrence  is a robust property. Indeed, it was proved by Cyr and Sarig \cite[Theorem 2.2]{cs} that the space of strongly positive recurrent potentials is open and dense (with respect to the uniform metric) in the space of weakly H\"older potentials with finite pressure.

\subsection{Entropy at infinity}\label{sec:introh}
A fundamental consequence of Theorem \ref{thm:main}  is that a great deal of dynamical information of the system is captured by its complexity at infinity. As discussed in the introduction, we have defined two different ways of quantifying this complexity. Namely, the topological entropy at infinity (Definition \ref{def:ent_inf}) and the measure theoretic one (Definition \ref{def:ent_meas_inf}). 
In this section we will elaborate on these notions and put our results into context.

We first discuss the topological entropy at infinity of $(\Sigma,\sigma)$, given in Definition \ref{def:ent_inf}.  Observe that if $M_1<M_2$, then $z_n(M_2, q)\le z_n(M_1, q)$, 
 for every $n, q \in \N$, so
\begin{equation*}
\delta_\infty(q)=\inf_M \delta_\infty(M,q)=\lim_{M\to\infty}\delta_\infty(M,q).
\end{equation*}
If $(\Sigma,\sigma)$ is a transitive CMS then for every $a,b\in \N$,
\begin{equation}
\label{eq:deltainfagain}
\delta_\infty(M,q)=\limsup_{n\to\infty}\frac{1}{n}\log z_n(M,q,a,b),
\end{equation}
where $z_n(M,q,a,b)$ is the number of cylinders of the form $[x_0,\ldots,x_{n+1}]$, where $x_0=a$, $x_{n+1}=b$, and 
\begin{equation*}
\# \left\{i\in\{0,\ldots,n+1\}:x_i\le q \right\}\le \frac{n+2}{M}.
\end{equation*}
Note that $q<q'$ implies the inequality $z_n(M,q',a,b)\le z_n(M,q,a,b)$. In particular $(\delta_\infty(M,q))_q$ is decreasing in $q$. We conclude that 
\begin{align}\label{infinf}
\delta_\infty=\inf_{q} \delta_\infty(q)=\inf_{M,q}\delta_\infty(M,q).
\end{align}

Since in our results we will usually  assume that the symbolic space is transitive, we can consider \eqref{infinf} as 
the definition of the topological entropy at infinity.  

We now consider the measure theoretic entropy at infinity, defined for finite entropy CMS as 
\begin{equation*}
h_\infty :=\sup_{(\mu_n)_n\to 0}\limsup_{n\to\infty}h_{\mu_n}(\sigma), 
\end{equation*}
where $(\mu_n)_n\to 0$ means that the sequence $(\mu_n)_n$ converges on cylinders to the zero measure.  Note that the finite entropy assumption, and more generally the $\F-$property, ensures the existence of sequences of measures converging on cylinders to the  zero measure (see \cite[Lemma 4.16]{iv}). In particular,  $h_{\infty}$ is well defined. 
In \cite[Example 4.17]{iv}, an example of a CMS made of infinitely many loops of length two based at a common vertex is considered. The entropy is infinite and there are no sequences of measures converging to zero. Every measure gives weight at least    $1/2$ to the base cylinder.
\label{rmk:finite ent}

 In Section \ref{entinf1} we will prove that both, the topological and the measure theoretic entropies at infinity coincide.  This has several consequences,  in particular we obtain that Theorem \ref{thm:main} is sharp. Indeed, $\delta_{\infty}$ is the smallest number for which inequality \eqref{eq:main} holds.

In the context of CMS the entropy at infinity was already investigated by  Gurevich and Zargaryan \cite{gz}, Ruette \cite{r} and Buzzi \cite{b}. It is important to mention that they also had two flavours of entropy at infinity, a topological and a measure theoretic version. It is proven by Ruette \cite{r} that both notions coincide (for a precise statement see \cite[Proposition 6.1]{b}). It turned out that the notions of entropy at infinity presented in this work coincide with theirs. Recall that if $G$ is the graph which determines $(\Sigma,\sigma)$, then
 \begin{equation*}
b_{\infty}=\inf_{F} \inf_{\lambda >0} \sup \left\{h_{\mu}(\sigma) : \mu([F]) < \lambda 			\right\},
\end{equation*} 
where $F$ ranges over the finite sub-graphs of $G$ and $[F]:= \left\{ x \in \Sigma : x_0 \in \mathcal{A}_F \right\}$, where $ \mathcal{A}_F$ denotes the symbols appearing as vertex of $F$. We first show the relation between $h_\infty$ and $b_\infty$.

\begin{lemma}
For a sequence $(\mu_n)_n$  in $\M(\Sigma, \sigma)$, the following are equivalent:
\begin{itemize}
\item[(a)]  for any collection of cylinders $C^1, \ldots, C^N,$ and $\eps>0$, there is $n_0\in \N$ such that $\mu_n(\bigcup_{i=1}^NC^i)<\eps$ for all $n\ge n_0$;
\item[(b)]  for any finite subgraph $F$ of $G$ and any $\eps>0$, there is $n_1\in \N$ such that  $\mu_n([F])<\eps$ for all $n\ge n_1$.
\end{itemize}
\label{lem:cylBuz}
\end{lemma}

An easy consequence of the lemma is that convergence on cylinders in this setting corresponds to the type of limits featuring in the definition of $b_\infty$ and thus $b_\infty=h_\infty$.

\begin{proof}[Proof of Lemma~\ref{lem:cylBuz}]
Since (b) concerns 1-cylinders, the fact that (a) implies (b) is clear.  To prove the reverse implication, we observe that if $C^1, \ldots, C^N$ is a collection of cylinders then we can take the subgraph defined by the first coordinate of each $C^i$ as our subgraph. 
\end{proof}

As previously mentioned, in Section \ref{entinf1} we will prove that $h_\infty=\delta_\infty$. This implies that the entropy at infinity defined in this section coincides with the previously defined one. One consequence is  that, since \cite[Proposition 6.1]{b} implies  $b_{\infty}<h_{top}(\sigma)$ is a characterisation of SPR, we thus have the following alternative characterisation:

\begin{proposition}\label{prechar} A topologically transitive CMS $(\Sigma,\sigma)$ is SPR if and only if $h_{\infty}<h_{top} (\sigma)$ if and only if $\delta_{\infty}<h_{top} (\sigma)$.\end{proposition}

This result is consistent with the comments in Remark \ref{rem:spr}. Indeed, SPR systems are those for which the entropy is not concentrated at infinity; the inequality $\delta_{\infty}<h_{top} (\sigma)$ has a wealth of dynamical consequences (see Remark \ref{rem_conspr}).

From a slightly different point of view, it was not realised until recently that the entropy at infinity has a particularly important role in the regularity of the entropy map. In the context of homogeneous dynamics, for the diagonal action on $G/\Gamma$, where $G$ is a $\R$-rank 1 semisimple Lie group with finite centre and $\Gamma\leqslant G$ a lattice, a formula like Theorem \ref{thm:main} was obtained in \cite{ekp}. In that context the constant playing the role of the entropy at infinity is half the topological entropy of the flow. It was later proved in \cite{kp} that half the topological entropy is in fact sharp and equal to the measure theoretic entropy at infinity in that setup. The method employed in \cite{ekp} was used in \cite{rv} to prove that a similar result holds for the geodesic flow on a geometrically finite manifold. Unfortunately, an obstruction to run the method from  \cite{ekp} is the existence of periodic orbits that escape to infinity. This issue was overcome in \cite{ve}, where the results in \cite{rv} where generalised to all complete negatively curved manifolds. For CMS the existence of periodic orbits that escape phase space is quite common so our approach is similar to the one in \cite{ve}. Additional complications arise from the possible lack of locally compactness of $\Sigma$. In Section \ref{mainine} and Section \ref{finalproof} we will address these issues and prove Theorem \ref{thm:main}. 

The entropy at infinity has further applications to suspension flows, entropy density, the dimension of points which escape on average, existence of equilibrium states and bounds on mass escape, all of which we give in Section~\ref{sec:app}.

\section{Katok's entropy formula } \label{ent}

In the early 1970s Bowen \cite{bo} and Dinaburg \cite{d} provided a new definition of topological entropy of a dynamical system. Inspired by these results, Katok  \cite{ka} established a formula  for the measure theoretic entropy in analogy  to the definition of topological entropy by Bowen and Dinaburg. We now recall his result in a particular context. 

Let $(\Sigma, \sigma)$ be a CMS and  let $d$ be the metric on $\Sigma$ defined in \eqref{metric}. The dynamical metric $d_n$ is defined by the formula
\begin{equation*}
d_n(x,y):=\max_{k\in \{0,\ldots,n-1\}}d(\sigma^kx,\sigma^ky).
\end{equation*}
The open ball of radius $r$ centred at $x$ with respect to the metric $d_n$ is denoted by $B_n(x,r)$. By the definition of the metric $d$ we know that $B_n(x, 2^{-N})=C_{n+N}(x)$. A ball of the form $B_n(x,r)$ is called a $(n,r)$-dynamical ball. The following result is a particular case of a theorem proved in \cite[Theorem 1.1]{ka}.

\begin{theorem}\label{kat} Let $(\Sigma,\sigma)$ be a sub-shift of finite type defined on a finite alphabet  and $\mu$  an ergodic $\sigma$-invariant probability measure. Then 
\begin{equation} \label{eq:kat}
h_\mu(\sigma)=\lim_{\epsilon\to 0} \liminf_{n\to \infty}\dfrac{1}{n}\log N_\mu(n,\epsilon,\delta),
\end{equation}
where $N_\mu(n,\epsilon,\delta)$ is the minimum number of $(n,\epsilon)$-dynamical balls needed to cover a set of $\mu$-measure strictly bigger than $1-\delta$. In particular the limit above does not depend on $\delta\in (0,1)$. 
\end{theorem}

The relation established in \eqref{eq:kat} is known as Katok's entropy formula. It turns out that Katok's proof is rather flexible. It was observed by  Gurevich and S. Katok  \cite[Section 4]{gk} and also by Riquelme \cite[Theorem 2.6]{ri} that the proof in \cite[Theorem 1.1]{ka} yields that if $(X,d)$ is a metric space (not necessarily compact) and $T: X \to X$ a continuous map then
\begin{equation*}
h_{\mu}(T) \leq  \lim_{\epsilon \to 0} \liminf_{n \to \infty} \frac{\log N_{\mu}(n, \epsilon ,\delta)}{n}.
\end{equation*}
The compactness assumption on $X$ is used in the proof of the other inequality. It is routine to check that compactness assumption can be replaced by the existence of a totally bounded metric.

This section is devoted  to proving that  formula \eqref{eq:kat} holds for CMS of finite topological entropy. Moreover, we will prove the limit is independent of $\e$. 
We prove:

\begin{theorem} \label{katformula}  Let $(\Sigma, \sigma)$ be a CMS  and $\mu$ an ergodic $\sigma$-invariant probability measure. Then for every $\delta\in (0,1)$ we have 
$$h_\mu(\sigma)\le \lim_{n\to \infty}\frac{1}{n}\log N_\mu(n,1,\delta).$$
If $(\Sigma,\sigma)$ has finite topological entropy, then 
$$h_\mu(\sigma)= \lim_{n\to \infty}\frac{1}{n}\log N_\mu(n,1,\delta).$$
\end{theorem}

Define the  following collection of sets: for every $m \in \N$ let
\begin{equation} \label{def:km}
K_m:=\Sigma\cap \bigcup_{s=1}^m [s].
\end{equation}
Note that if $\Sigma$ is locally compact, then $K_m$ is compact for every $m\in \N$.  To every sequence of natural numbers $(a_i)^\infty_{i=0}$ we associate the set  
\begin{equation} \label{def:kai}
K((a_i)_i)= \Sigma\cap \prod_{i\ge 0} \{1,\ldots,a_i\}.
\end{equation}
Observe that $K((a_i)_i)$ is the intersection of a closed set with a compact set and is thus a compact subset of $\Sigma$. Moreover, every compact set $K\subset \Sigma$ is contained in a set of the form $K((a_i)_i)$. The following lemma follows directly from \cite[Theorem 3.2]{pa}. For concreteness we provide a simple proof of this general fact. 

\begin{lemma}\label{lem:compact} Let $\mu$ a Borel measure on $\Sigma$. For every $\e>0$, there exists a sequence of natural numbers  $(a_i)_i$ such that $\mu(K((a_i)_i))>1-\e$.
\end{lemma}

\begin{proof} Fix a sequence $(b_i)_i$ satisfying 
\begin{equation*}
\left(1-\frac{\e}{2}\right)\prod_{i=1}^\infty b_i> 1-\e,
\end{equation*}
where $b_i\in (0,1)$ for every $i\in\N$. We will construct the sequence $(a_i)_i$ inductively. Choose $a_0$ such that $\mu(\bigcup_{i=1}^{a_0} [i])>1-\frac{\e}{2}$. For every $i\in \{1,\ldots,a_0\}$ we choose $c(i)\in \N$ such that 
\begin{equation*}
\mu\left(\bigcup^{c(i)}_{k= 1}[ik]\right)\ge \mu([i])b_1.
\end{equation*}
Let $a_1:=\max_{i\in \{1,\ldots,a_0\}} c(i)$. For $(i_1,i_2)\in \prod_{i=0}^1\{1,\ldots,a_i\}$ we define $c(i_1,i_2)$ such that 
\begin{equation*}
\mu\left(\bigcup^{c(i_1,i_2)}_{k= 1}[i_1i_2k]\right)\ge \mu([i_1i_2])b_2.
\end{equation*}
Define $a_2=\max_{(i,j)\in \prod_{i=0}^1 \{1,\ldots,a_i\} } c(i,j)$. We continue this procedure inductively. It follows from the construction that 
\begin{equation*}
\mu(K((a_i)_i))=\mu\left(\prod_{i=0}^\infty \{1,\ldots,a_i\} \right)\ge \left(1-\frac{\e}{2}\right)\prod_{i=1}^\infty b_i>1-\e,
\end{equation*}
as desired. 
\end{proof}

\begin{remark}\label{rem:upbound} Katok proved   \cite[Theorem 1.1]{ka} that if  $\P$ is any finite  partition of $\Sigma$ satisfying $\mu(\partial \P)=0$, then for any $\delta \in (0,1)$
\begin{equation*}
h_\mu(\P)\le \lim_{r\to 0}\liminf_{n\to\infty}\frac{1}{n}\log N_\mu(n,r,\delta).
\end{equation*}
For a CMS it is easy to check that the partitions 
\begin{equation*}
\P_n=\left\{[1],\ldots,[n],\bigcup_{s>n}[s]\right\}
\end{equation*}
are such that $\partial \P_n=\emptyset$, and $\lim_{n\to\infty}h_\mu(\sigma,\P_n)=h_\mu(\sigma)$. From this we conclude that
\begin{equation*}
h_\mu(\sigma)\le \lim_{r\to 0}\liminf_{n\to\infty}\frac{1}{n}\log N_\mu(n,r,\delta).
\end{equation*}
\end{remark}

Our next result is inspired by the proof of \cite[Theorem 2.10 and Theorem 2.11]{ri}. In our context we do not have local compactness of $\Sigma$: the finite entropy assumption is important in overcoming this issue.

\begin{lemma}\label{lem:katokeq} Let $(\Sigma,\sigma)$ be a CMS with finite topological entropy. If $\mu$ is an ergodic $\sigma$-invariant probability measure, then for every $\delta\in (0,1)$ we have
\begin{equation*}
h_\mu(\sigma)=\lim_{N\to \infty} \liminf_{n\to \infty}\frac{1}{n}\log N_\mu(n,2^{-N},\delta).
\end{equation*}
\end{lemma}

\begin{proof}
As observed in  Remark \ref{rem:upbound} the inequality 
\begin{equation*}
h_\mu(\sigma)\le\lim_{N\to \infty} \liminf_{n\to \infty}\frac{1}{n}\log N_\mu(n,2^{-N},\delta)
\end{equation*}
 is known to hold.  For the converse inequality it suffices to prove that for every $\ell\in \N$ there exists a partition $\P=\P(\ell)$ of $\Sigma$ and a subset $ K \subset \Sigma$ satisfying:
 \begin{enumerate}
 \item The partition $\P(\ell)$ has finite entropy with respect to $\mu$.
 \item  $\mu(K)>1-\frac{\delta}{6}$.
 \item  For every $x\in K\cap \sigma^{-n}K$ we have $\P^n(x)\subset B_n(x,2^{-\ell})$.
 \end{enumerate}
In this situation a slight modification of the first part of the proof in \cite[Theorem 1.1]{ka} yields the desired inequality,  as we show here. 
Suppose that the partition $\P=\P(\ell)$ has been constructed. Let $\e >0$. Since the measure $\mu$ is ergodic  by the Shannon-McMillan-Breiman theorem there exists $N_0 \in \N$ such that the set
\begin{equation*}
A_{\e, N_0}:=\left\{x\in \Sigma: \mu(\P^n(x))\ge \exp(-n(h_\mu(\P)+\e)), \text{for all }n\ge N_0 \right\}.
\end{equation*}
satisfies $\mu(A_{\e,N_0})>1-\frac{\delta}{6}$. Let  $n\ge N_0$ and $B_n:=A_{\e,N_0}\cap K\cap \sigma^{-n}K$. Observe that $\mu(B_n)\ge 1-\frac{\delta}{2}$ and that  if $x\in B_n$ then $x\in K\cap \sigma^{-n}K$, and therefore $\P^n(x)\subset B_n(x,2^{-\ell})$. The set $A_{\e,N_0}$ requires at most $\exp(n(h_\mu(\P)+\e))$ elements of the partition $\P^n$ to cover it. Therefore,  $B_n$ requires at most  $\exp(n(h_\mu(\P)+\e))$ $(n,2^{-\ell})$-dynamical balls to cover it, where $\mu(B_n)>1-\frac{\delta}{2}$. We conclude that
\begin{equation*}
\limsup_{n\to\infty}\frac{1}{n}\log N_\mu(n,2^{-\ell},\delta)\le h_\mu(\P)+\epsilon\le h_\mu(\sigma)+\epsilon.
\end{equation*}
Since $\epsilon>0$ was arbitrary we obtain
\begin{equation*} 
\lim_{\ell\to\infty}\limsup_{n\to\infty}\frac{1}{n}\log N_\mu(n,2^{-\ell},\delta)\le h_\mu(\sigma),
\end{equation*}
concluding the proof of the lemma.

We now prove the existence of such a  partition $\P=\P(\ell)$.  By Lemma \ref{lem:compact} there exists a sequence $(a_i)_i$ such that the compact set  $K_0:=K((a_i)_i)$ satisfies $ \mu(K_0) \ge 1-\frac{\delta}{6}$. 
Denote by  $S$ the set of points in $\Sigma$ that enter $K_0$ infinitely many times under iterates of $\sigma$. It is a consequence of  Birkhoff's Ergodic Theorem that $\mu(S)=1$.  Define $K:=K_0\cap S$, and observe that $\mu(K)\ge 1-\frac{\delta}{6}$. For every $k \geq 1$, let  
\begin{equation*}
R_K(x):=\inf\left\{k\ge 1: \sigma^k(x)\in K_0\right\} \text{ for } x\in K, \text{ and }
A_k:=\left\{x \in K: R_K(x)=k	\right\}.
\end{equation*}
 Partition $A_k$ using cylinders of length $k+\ell+1$ and  denote such partition by $\Q_k$. It is important to observe that $\#\Q_k$ is finite for all $k$. This follows from the definition of $K_0$ and the finite topological entropy of $(\Sigma,\sigma)$.  Indeed, if   $x= (x_0, x_1, \dots , x_k \dots) \in A_k$, then $x_0 ,x_k \in \{1,\ldots,a_0\}$.  Moreover, there are at most $C=\prod_{i=0}^{\ell} a_i$ cylinders of the form $[y_0 y_1\ldots y_l]$ intersecting $K$, so for $k$ large enough, $$\#\Q_k\le Ce^{k(h_{top}(\sigma)+1)}.$$ 

Finally, consider the partition of $\Sigma$ defined by  $\P= \mathcal{Q} \cup \left( \bigcup_{k=1}^{\infty} \Q_k \right)$, where $\mathcal{Q}:= \Sigma \setminus \bigcup_{k=1}^{\infty} \Q_k$.
We claim that this countable partition satisfies the remaining required properties, that is: 
\begin{enumerate}
 \item The partition $\P=\P(\ell)$ has finite entropy with respect to $\mu$.
 \item  For every $x\in K\cap \sigma^{-n}K$ we have $\P^n(x)\subset B_n(x,2^{-\ell})$.
 \end{enumerate} 
The second property follows from the construction of $\P$. Indeed, let $z\in \P^n(x)$, where $x , \sigma^n(x) \in K$. We claim that  $z\in B_n(x,2^{-\ell})$. For simplicity we will assume that $x$ has its first return to $K$ at time $n$ (the general case is just an iteration of the argument in this setting). Since $x \in A_n$ we have that $\P(x)$ is a  cylinder of length $n+\ell+1$, which readily implies that $z\in B_n(x,2^{-\ell})$.

We now verify that $H_\mu(\P)<\infty$. For $r$ sufficiently large, 
\begin{align*} H_\mu(\P)+R & =  \sum_{k\ge r} \sum_{P\in \Q_k}-\mu(P)\log \mu(P)\\
&= \sum_{k\ge r} \mu(A_k)\left( \sum_{P\in \Q_k}-\frac{\mu(P)}{\mu(A_k)}\log \frac{\mu(P)}{\mu(A_k)}-\frac{\mu(P)}{\mu(A_k)} \log\mu(A_k)\right) \\ 
 &\le\sum_{k\ge r}\mu(A_k)\log(|\Q_k|)-\sum_{k\ge r}\mu(A_k)\log \mu(A_k) \\
&\le \sum_{k\ge r}k\mu(A_k)\log\left(e^{h_{top}(\sigma)+1}C^{1/k}\right)-\sum_{k\ge r}\mu(A_k)\log \mu(A_k)\\
&\le C'\sum_{k\ge r}k\mu(A_k)-\sum_{k\ge r}\mu(A_k)\log \mu(A_k),
\end{align*}
where $R=\mu(\Q)\log\mu(\Q)+\sum_{k=1}^{r-1}\sum_{P\in \Q_k}\mu(P)\log\mu(P)\in \R$. It follows from  Kac's lemma that $\sum k\mu(A_k)=1$. This and the inequality  
\begin{align*}\label{mane}
-\sum_{k\ge r}\mu(A_k)\log \mu(A_k)\le \sum_{k\ge r}k\mu(A_k)+2e^{-1}\sum_{k\ge r}e^{-k/2},
\end{align*}
see \cite[Lemma 1]{m}, imply  the finiteness of $H_\mu(\P)$. This concludes the proof.
\end{proof}

\begin{lemma}\label{lem:katokine} Let $(\Sigma, \sigma)$ be a CMS and $\mu$ an ergodic $\sigma$-invariant probability measure. Then for any  $\delta\in (0,1)$, we have
\begin{equation*}
h_\mu(\sigma)\le \liminf_{n\to\infty} \frac{1}{n}\log N_\mu(n,1,\delta).
\end{equation*}
\end{lemma}

\begin{proof}
 Let $A \subset \Sigma$ be a set such that $\mu(A) \geq 1 - \delta$. Denote by  $a_\mu(n,\delta)$ the minimum number of cylinders of length $n$ that cover a $A$. Observe that 
\begin{equation*}
 N_\mu(n,2^{-t},\delta)=a_\mu(n+t,\delta),
 \end{equation*}
and that
\begin{equation*}
\liminf_{n\to\infty}\frac{1}{n}\log a_\mu(n,\delta)=\liminf_{n\to\infty}\frac{1}{n}\log a_\mu(n+t,\delta),
\end{equation*}
for every $t\in\N$. In particular we have that $\liminf_{n\to\infty}\frac{1}{n}\log N_\mu(n,2^{-\ell},\delta)$ is independent of $\ell$.  
From Remark \ref{rem:upbound} we conclude that 
\begin{align*}h_\mu(\sigma)&\le \lim_{t\to \infty}\liminf_{n\to \infty}\frac{1}{n}\log N_\mu(n,2^{-t},\delta)\\ 
&=\lim_{t\to \infty}\liminf_{n\to \infty}\frac{1}{n}\log a_\mu(n+t,\delta)\\
&=\liminf_{n\to \infty}\frac{1}{n}\log a_\mu(n,\delta).
\end{align*}
\end{proof}

\begin{proof}[Proof of Theorem \ref{katformula}] 
The proof follows combining Lemma \ref{lem:katokeq} and Lemma \ref{lem:katokine}. 
\end{proof}

We now prove a result related to Lemma \ref{lem:katokine}. We say that two points $x, y \in \Sigma$ are $(n,r)$-separated if $d_n(x,y)\ge r$. In particular $x$ and $y$ are $(n,1)$-separated if they do not belong to the same cylinder of length $n$.

\begin{lemma}\label{lem:compshift}  Let $X$ be a $\sigma$-invariant compact subset of $\Sigma$. Then 
\begin{equation*}
h_{top}(\sigma |X)=\limsup_{n\to\infty} \frac{1}{n}\log N(X,n),
\end{equation*}
where $N(X,n)$ is the maximal number of $(n,1)$-separated points in $X$, and $h_{top}(\sigma|X)$ is the topological entropy of $(X,\sigma)$. 
\end{lemma}

\begin{proof} By definition of the topological entropy of a compact metric space we know that  
\begin{equation*}
h_{top}(\sigma |X)=\lim_{k\to\infty}\limsup_{n\to\infty}\frac{1}{n}\log N(X,n,k),
\end{equation*}
where $N(X,n,k)$ is the maximal number of $(n, 2^{-k})$-separated points in $X$. Observe that being $(n,2^{-k})$-separated is the same as being $(n+k,1)$-separated. This implies that $N(X,n,k)=N(X,n+k)$. Note that 
\begin{equation*}
\limsup_{n\to\infty}\frac{1}{n}\log N(X,n,k)=\limsup_{n\to\infty}\frac{1}{n}\log N(X,n+k)=\limsup_{n\to\infty}\frac{1}{n}\log N(X,n).
\end{equation*}
Therefore 
\begin{equation*}
h_{top}(\sigma |X)=\lim_{k\to\infty}\limsup_{n\to\infty}\frac{1}{n}\log N(X,n,k)=\limsup_{n\to\infty} \frac{1}{n}\log N(X,n),
\end{equation*}
as desired.
\end{proof}

\section{Weak entropy density} \label{sec:wed}

In this section we describe the  inclusion $\mathcal{E}(\Sigma, \sigma) \subset \M(\Sigma,\sigma)$, where $\mathcal{E}(\Sigma, \sigma)$ is the subset of ergodic measures. It is well known that, even in this non-compact setting, the set $\mathcal{E}(\Sigma, \sigma)$ is dense in $ \M(\Sigma,\sigma)$ with respect to the weak* topology (see \cite[Section 6]{csc}). We prove that  any finite entropy measure can be approximated by an ergodic measure with entropy sufficiently large, see Proposition \ref{dense}. This result can be thought of as a weak form of entropy density.  In Section \ref{sec:ed} we will make use of this result to prove that any invariant measure $\mu$ can be approximated by ergodic measures with entropy converging to $h_{\mu}(\sigma) $ (see Theorem  \ref{teodense}). Moreover, Proposition \ref{dense}  will be used in the proof of our main result (see Theorem \ref{thm:main}). Both the statement and the proof  of Proposition \ref{dense} closely follow that of \cite[Theorem B]{ekw}, but  modifications are required to deal with the non-compactness of the space $\Sigma$.

\begin{proposition} \label{dense} Let $(\Sigma,\sigma)$ be a  transitive  CMS. Then for every $\mu\in \M(\Sigma,\sigma)$ with $h_{\mu}(\sigma) <\infty$, $\epsilon>0$, $\eta>0$, and $f_1,\ldots,f_l\in C_b(\Sigma)$, there exists an ergodic measure $\mu_e\in V(f_1,\ldots,f_l,\mu,\epsilon)$ (see equation (\ref{defbasis}))  such that $h_{\mu_e}(\sigma)>h_\mu(\sigma) -\eta$. We can moreover assume that $\supp (\mu_e)$ is compact.
\end{proposition}

Analogously to the proof of  \cite[Theorem B]{ekw} we will use the following fact. 

\begin{lemma} \label{entropyforergodic} Let  $\mu\in \mathcal{E}(\Sigma, \sigma)$, $\alpha>0$, $\beta>0$, $f_1,\ldots,f_\ell\in C_b(\Sigma)$, and a set $K \subset \Sigma$ satisfying $\mu(K)>3/4$. Assume that $h_\mu(\sigma) <\infty$. Then there exists $n_0 \in \N$ such that for all $n\ge n_0$ there is a finite set $\G=\G(n)\subset\Sigma$ satisfying the following properties: 
\begin{enumerate}
\item  \label{a} $\G\subset K\cap \sigma^{-n}K$
\item  \label{b} $d(x,y)>2^{-n}$, for every pair of distinct points $x,y\in \G$. 
\item \label{c}  $\#\G\ge \exp(n(h_\mu(\sigma) -\alpha))$.
\item \label{d} $|\frac{1}{n}\sum_{k=0}^{n-1} f_j(\sigma^k x)-\int f_j d\mu|<\beta$, for all $x\in \G$ and $j\in\{1,\ldots,\ell\}$.
\end{enumerate}
\end{lemma} 

\begin{proof} Let $$A_{k,\beta}:=\left\{x\in \Sigma: \left|\sum_{i=0}^{n-1} f_j(\sigma^i x)-\int f_j d\mu\right|<\beta, \forall j\in \{1,\ldots,\ell\} \text{ and }n\ge k\right\}.$$
 By Birkhoff's Ergodic Theorem there exists $s_0 \in \N$ such that  $\mu(A_{s_0,\beta})>3/4$. From Lemma \ref{lem:katokine} we have that
 \begin{equation*}
 h_\mu(\sigma) \le \liminf_{n\to\infty} \frac{1}{n}\log N_\mu(n,1,1/4).
\end{equation*} 
There exists $s_1 \in \N$ such that if $n\ge s_1$, then 
\begin{equation*}
\exp(n(h_\mu(\sigma)-\alpha))\le N_\mu(n,1,1/4).
\end{equation*}
 
Let $B_n:=K\cap \sigma^{-n}K\cap A_{s_0,\beta}$ and observe that $\mu(B_n)>1/4$. In what follows we assume that $n\ge n_0:=\max\{s_0,s_1\}$.  From the definition of $N_\mu(n,1,1/4)$  the minimal number of cylinders of length $n$ needed to cover $B_n$ is at least $N_\mu(n,1,1/4)$. More precisely, let $(C_i)_{i\in I}$ be a minimal collection of cylinders of length $n$ covering $B_n$. In particular for every $i\in I$ we have  $C_i\cap B_n\ne \emptyset$. For every $i\in I$ choose a point $x_i\in C_i\cap B_n$.  We claim that the set $(x_i)_{i\in I}$ satisfies the  properties required on $\G$. Conditions \eqref{a} and \eqref{d}  follow from the definition of $B_n$. Condition \eqref{b} follows from the fact that if $i\ne j$, then $x_i$ and $x_j$ are in different cylinders of length $n$. Condition \eqref{c} follows from the inequality $$\#I\ge N_\mu(n,1,1/4)\ge \exp(n(h_\mu(\sigma)-\alpha)).$$ 
\end{proof}

\begin{proof}[Proof of Proposition~\ref{dense}] 
Recall that we want to prove that given $\mu\in \M(\Sigma,\sigma)$, $\epsilon>0$, $\eta>0$, and $f_1,\ldots,f_l\in C_b(\Sigma)$, there exists an ergodic measure $\mu_e\in V(f_1,\ldots,f_\ell,\mu,\epsilon)$ such that $h_{\mu_e}(\sigma)>h_\mu(\sigma) -\eta$.  In the following remarks we observe that this general situation can be  simplified.

As observed in Section \ref{weak*} or in   \cite[8.3.1 Remark]{bg} it suffices to consider the case in  which  the functions  $(f_i)_i$ in Proposition \ref{dense} are uniformly continuous. Therefore, under this assumption, there exists $A=A(f_1,\ldots,f_\ell)\in \N$,  such that if $d(x,y)<2^{-A}$, then $|f_i(x)-f_i(y)|<\frac{\epsilon}{4}$. Also define $W=\max_{i\in \{1,...,\ell\}}|f_i|_0$.  

Since periodic measures are dense in  $\M(\Sigma,\sigma)$, see \cite[Section 6]{csc}, we will assume that  $h_\mu(\sigma) -\eta>0$, otherwise we can approximate $\mu$ by a periodic measure. By the affinity of the entropy map \cite[Theorem 8.1]{wa}  and  \cite[Lemma 6.13]{iv}
we can reduce the problem to the case in which $\mu=\frac{1}{N}\sum_{i=1}^N \mu_i$, where $\{\mu_i\}_{i=1}^N$ is a collection of ergodic measures.

Let  $m \in \N$ be such that the set $K=K_m$, defined as in \eqref{def:km}, satisfies $\mu_i(K)>3/4$ for every $i \in \{1, \dots, N\}$. 
Since $(\Sigma, \sigma)$ is transitive, there exists a constant $L=L(m)$ such that for each pair $(a,b)\in \{1,\ldots,m\}^2$,  there exists an admissible word $a{\bf  r }b$, where $\ell({\bf r})\le L$.  It follows from Lemma \ref{entropyforergodic}, setting  $\beta=\epsilon/4$ and $\alpha=\eta/2$, that there exists $n' \in \N$ such that for every $n >n'$ and  every measure $\mu_i$, with $i \in \{1, \dots , N\}$, there exists $(n,1)$-separated sets $\G_i\subset K\cap \sigma^{-n}K$ satisfying properties \eqref{a}, \eqref{b}, \eqref{c} and \eqref{d} of Lemma \ref{entropyforergodic}. 
 
Denote by  ${\bf a}(x)$  the word defined concatenating the first $(n+1)-$coordinates of $x \in \Sigma$. Given  $\hat{x}=(x^1,x^2,\ldots,x^{MN})\in (\prod_{i=1}^N \G_i)^M$, we define an admissible word $w_0(\hat{x})={\bf a}(x^1) {\bf r_1} {\bf a}(x^2){\bf r_2}\ldots{\bf a}(x^{MN}){\bf r_{MN}}{\bf a}(x^1)$, where ${\bf r_k}$s are words chosen so that  $w_0(\hat{x})$ is an admissible word and $\ell({\bf r_k})\le L$ (note that this is possible since $({\bf a}(x^i))_{0}$ and $({\bf a}(x^i))_{n}$ are in $\{1,\ldots,m\}$ by definition of $\G_i$).
The word $w_0(\hat{x})$ defines a periodic point in $\Sigma$ that we denote by  $w(\hat{x})$. We have that
\begin{equation*}
w(\hat{x})=\overline{{\bf a}(x^1) {\bf r_1} {\bf a}(x^2){\bf r_2}\ldots{\bf a}(x^{MN}){\bf r_{MN}}}.
\end{equation*}

Let  $\G:=\prod_{M\ge 1} (\prod_{i=1}^N \G_i)^M$. Following the same procedure of concatenation described above, for every $\hat{x} \in \G$ we define a  point $w(\hat{x})\in \Sigma$. Define 
\begin{equation*}
\Psi=\bigcup_{\hat{x}\in \G}O(w(\hat{x})), 
\end{equation*}
where $O(w(\hat{x}))$ is the orbit of $w(\hat{x})$ and define $\Psi_0$  to be the topological closure of $\Psi$. 

Note that the space $\Psi_0$ is a compact $\sigma$-invariant  subset of $\Sigma$. By definition the set $\Psi$ is closed and invariant.
Observe that the number of symbols appearing in elements belonging  to $\Psi$ is finite: there are finitely many admissible words ${\bf a}(x^i)$ (recall that each $\G_i$ is a finite set) and we could use finitely many connecting words ${\bf r_i}$. Therefore  there exists $J \in \N$ such that $\Psi \subset \{1,\ldots,J\}^\N$. Thus, $\Psi_0$ is a closed subset of a compact set. 

By property \eqref{d} of Lemma \ref{entropyforergodic}, and assuming that $n'$, which also depends on $A$, $W$ and $L$,  is sufficiently large, 
\begin{align} \label{eq:incl}
\Psi&\subset  \left\{x\in \Sigma: \left|\frac{1}{n}S_n f_j(x)-\int f_j d\mu \right|\le\epsilon,\forall j\in\{1,\ldots,\ell\} \right\}.
\end{align}
Since the set in right hand side of \eqref{eq:incl} is closed, the same inclusion holds if $\Psi$ is replaced by $\Psi_0$. Also, since $\Psi_0$ is $\sigma$-invariant we have 
\begin{align*} \Psi_0&\subset  \left\{x\in \Sigma: \left|\frac{1}{n}S_n f_j(\sigma^s x)-\int f_j d\mu \right|\le\epsilon, \forall  j\in\{1,\ldots,\ell\}\text{ and }s\ge 0 \right\}\\
& \subset  \left\{x\in \Sigma: \left|\frac{1}{nk}S_{nk} f_j( x)-\int f_j d\mu \right|\le\epsilon,  \forall  j\in\{1,\ldots,\ell\}\text{ and }k\in \N \right\}.
\end{align*}
This implies that every ergodic measure supported in $\Psi_0$ belongs to $V(f_1,\ldots,f_l,\mu,\epsilon)$. Indeed, consider a generic point for the ergodic measure and use the inclusion above. 

By construction, if $x,y\in \left(\prod_{i=1}^N \G_i\right)^M$ and $x\ne y$, then $$d_{NM(n+1+L)}(w(x),w(y))=1.$$ In other words $\Psi_0$ contains a $(NM(n+1+L),1)$-separated set of cardinality at least
\begin{equation*}
\exp\left(nNM \left(\frac{1}{N}\sum_{k=1}^N h_{\mu_k} (\sigma)-\frac{\eta}{2} \right) \right).
\end{equation*}
Here we used property \eqref{c} of Lemma \ref{entropyforergodic} for our sets $\G_i$. It follows from Lemma \ref{lem:compshift}  that
\begin{equation*}
h_{top}(\Psi_0)\ge \limsup_{M\to \infty} \dfrac{nNM(h_\mu(\sigma)-\frac{\eta}{2})}{NM(n+1+L)}= \dfrac{n(h_\mu(\sigma)-\frac{\eta}{2})}{(n+1+L)} > h_\mu(\sigma)-\eta.
\end{equation*}

Finally, let $\mu_e$ be an ergodic measure supported in $\Psi_0$ with entropy at least $h_\mu(\sigma)-\eta$ (which exists by the standard variational principle  in the compact setting), since we  already  proved that  $\mu_e\in V(f_1,\ldots,f_l,\mu,\epsilon)$  this finishes the proof. 
\end{proof}

\section{Main entropy inequality}\label{mainine}
This section is devoted to the proof of the main entropy  inequality. This is stated  in Theorem \ref{pre} and relates the entropy of a sequence of ergodic measures with the amount of mass lost and the topological entropy at infinity.

Recall that, as explained in \eqref{def:kai}, to every sequence  of natural numbers  $(a_i)_i$ we assign a compact set $K=K((a_i)_i) \subset \Sigma$.
The  definition of $K$ implies that if $x\in K^c$, then $x_i>a_i$, for some $i\in \N_0$. For $x\in K^c$ we define $i: K^c \to \N_{0}$ by
\begin{equation} \label{def:i}
i(x):= \min \left\{	n \in \N_{0} :  x_n > a_n		\right\}.
\end{equation}
For $n\in \N$ we define 
\begin{equation} \label{def:T}
T_n(K):=K_{a_0}\cap \sigma^{-1}K^c\cap\cdots\cap\sigma^{-n}K^c\cap \sigma^{-(n+1)}K_{a_0},
\end{equation}
where $K_{a_0}=\bigcup_{i=1}^{a_0}[i]$ (as defined in \eqref{def:km}). Let
\begin{equation}   \label{def:That}
\widehat{T}_n(K):=\left\{x\in T_n(K): i(\sigma^k(x))\le n-k,\text{ for every } k\in \{1,\ldots,n\}  \right\}.
\end{equation}
Let $\hat z_n(K)$ be the minimal number of cylinders of length $(n+2)$ needed to cover $\widehat{T}_n(K)$ and define 
\begin{equation} \label{def:di}
\hat\delta_\infty(K):=\limsup_{n\to\infty}\frac{1}{n}\log \hat z_n(K).
\end{equation}
The reason why we define $\hat\delta_\infty(K)$ covering the sets $\widehat{T}_n(K)$, and not $T_n(K)$, is to ensure Lemma \ref{lem:Kineq2}. This allows us to relate $\hat\delta_\infty(K)$ with the topological entropy at infinity of $(\Sigma,\sigma)$. 

Our next result is fundamental in this paper.

\begin{theorem} \label{pre} Let $(\Sigma,\sigma)$ be a finite entropy CMS.  Let $(\mu_n)_n$ be a sequence of ergodic probability measures converging  on cylinders to an invariant measure $\mu$. Let $(a_i)_i$ be an increasing sequence of natural numbers such that the corresponding compact set  $K=K((a_i)_i)$ satisfies that $\mu_n(K)>0$, for all $n\in \N$.  Then 
\begin{equation*}
\limsup_{n\to \infty} h_{\mu_n}(\sigma)\le |\mu|h_{\mu/|\mu|}(\sigma)+(1-\mu(Y))\hat\delta_\infty(K),
\end{equation*}
where $Y=\bigcup_{s=0}^\infty \sigma^{-s}K$.
\end{theorem}

The proof of this theorem requires some propositions and lemmas, which we will prove first before completing the proof of the theorem at the end of this section.  

The fact that  $K \subset \sigma (K)$, which follows since $(a_i)_i$ is an increasing sequence, will be used several times here. Let $A_k:=\left\{ x \in K : R_K(x)=k \right\}$, where $R_K(x)$ is the first return time function to the set $K$ (see Definition \ref{def:ret}). For $x \in Y$ we define the following: 
\begin{align*}
n_1(x)&:= \min \left\{ n \in\N_0	: \text{ there exists } y \in K \text{ such that } \sigma^{n}(y)=x	\right\},\\ 
n_2(x)&:= \min \left\{ n \in\N_0	: \sigma^{n}x\in K	\right\}.
\end{align*}
We emphasise that the function $n_1(x)$ is well defined. Indeed, observe that if $x\in Y$ then $\sigma^{n_2(x)}x\in K$. Let $r \in \N$ be such that $r> n_2(x)-1$. Since the sequence $(a_i)_i$ is increasing we have that $a_r \geq \max\left\{ a_i :  i\in\{0,\ldots,n_2(x)-1 \} \right\}$. Since $x\in \sigma^r(K)=\prod^\infty_{k=r}\{1,\ldots,a_k\}\cap \Sigma$ we have that  $n_1(x)$ is finite for every $x\in Y$.  Let
\begin{eqnarray*}
n(x):=
\begin{cases}
n_1(x)+n_2(x) & \text{ if } x \in Y, \\
\infty & \text{ if } x\in \Sigma\setminus Y.
\end{cases}
\end{eqnarray*}
For $n\in\N_0 \cup \{\infty\}$ define 
\begin{equation*}
\mathcal{C}_n:=\{x\in\Sigma: n(x)=n\}.
\end{equation*}
Note that $\mathcal{C}_0=K$ and $\mathcal{C}_1=\emptyset$. For $n \geq 2$ observe that $x\in \mathcal{C}_n$ if it belongs to the orbit of a point in $A_n$. More precisely,  for every $n\ge 2$ we have that $\mathcal{C}_n\subset \bigcup_{k=1}^{n-1}\sigma^k(A_n)$. We define the following sets, 

\begin{equation*}
\alpha_{\le N}:=\left(\bigcup_{n=2}^{N} \mathcal{C}_n\right), \alpha_{N,M}:=\left(\bigcup^M_{n> N} \mathcal{C}_n\right)  \text{ and } \alpha_{>M}:=\left(\bigcup_{n> M} \mathcal{C}_n\right)\cup \mathcal{C}_\infty.
\end{equation*}

\begin{remark}
The set $\alpha_{\le M}$ can be covered with finitely many cylinders of length $L$. Indeed, observe that for every $n\ge 2$ we have
\begin{equation*}
\mathcal{C}_n \subset \bigcup_{s=1}^{n-1}\sigma^s(A_n) \subset \bigcup_{s=1}^{n-1}\sigma^s(K) \subset \sigma^{n-1}(K).
\end{equation*}
Therefore,
\begin{equation*}
\alpha_{\le M}=\bigcup_{n=2}^{M}\mathcal{C}_n \subset \sigma^{M-1}(K)=\prod^\infty_{s=M-1}\{1,\ldots,a_s\}\cap \Sigma.
\end{equation*}
Since the set  $\prod^\infty_{s=M-1}\{1,\ldots,a_s\}\cap \Sigma$ can be covered with at most $\prod_{s=M-1}^{M-2+L}a_s$   cylinders of length $L$,  the same holds for $\alpha_{\le M}$. 
\end{remark}

Observe that it follows directly from the definition of $\hat\delta_\infty(K)$  (see \eqref{def:di}) that for every  $\e>0$, there exists $N_0=N_0(\e) \in \N$ such that for every $n\ge N_0$  we have 
\begin{equation*}
\hat z_n(K)\le e^{n(\hat\delta_\infty(K)+\e)}.
\end{equation*}

At this point we fix $\epsilon >0$ and  $k , N \in \N$ large enough so that $kN\ge N_0(\e)$: these will appear explicitly in the proof of Theorem~\ref{pre}.

Given $A\subset \Sigma$ and $t\in \N$ we define
\begin{equation*}
U_t(A):=\left\{x\in \Sigma: d(x,A)\le 2^{-t} \right\}.
\end{equation*}
Now let
\begin{eqnarray*}
K(k,N):=U_{kN+2}(K) , &\\  \gamma_{\le N}:=U_{(k+1)N+2}(\alpha_{\le N})\setminus K(k,N) , &\\ G_{k,N}:=K(k,N)\cup \gamma_{\le N},
\end{eqnarray*}
and  
\begin{eqnarray*}
\gamma_{N,kN}:=U_{2(k+1)N+2}(\alpha_{N,kN})\setminus G_{k,N}, &\\  \gamma_{>kN}:=\Sigma\setminus (G_{k,N}\cup \gamma_{N,kN}), &\\ B_{k,N}:=\gamma_{N,kN}\cup\gamma_{>kN}.
\end{eqnarray*}
Denote by $\Q_1(k,N)$ the minimal cover of $K(k,N)$ with cylinders of length $kN+2$. Similarly, denote by $\Q_2'(k,N)$ the minimal cover of $\alpha_{\le N}$ with cylinders of length $(k+1)N+2$. Observe that every element in $\Q_2'(k,N)$ is disjoint or contained in an element of $\Q_1(k,N)$. In particular $\gamma_{\le N}$ is a finite union of cylinders of length $(k+1)N+2$; this collection of cylinders is denoted by $\Q_2(k,N)$.   Define 
\begin{equation} \label{def:beta'}
\beta_{k,N}':=\Q_1(k,N)\cup \Q_2(k,N)
\end{equation}
and observe that $\beta_{k,N}'$ is a partition of the set $G_{k,N}$. Define the following partition of $\Sigma$,
\begin{equation} \label{def:beta}
\beta_{k,N}:=\{\gamma_{>kN},\gamma_{N,kN}\}\cup  \beta_{k,N}'.
\end{equation}
Recall that the refinement $ \beta_{k,N}^n$ follows as in Section  \ref{sec:em}.  \\

\emph{Notation:} We use the following notation for an interval of integers $[a,b):= \{ n \in \N  : a  \leq n < b\}$   and $|[a, b)|=b-a$.

\begin{definition}  
Let $Q\in \beta_{k,N}^n$ be such that $(Q\cup \sigma^{n-1}{Q})\subset G_{k,N}$. An interval $[r,s)\subset [0,n)$ is called an \emph{excursion} of $Q$  into $\gamma_{>kN}$ (resp. $B_{k,N}$)  if $\sigma^t Q\subset   \gamma_{>kN}$ (resp. $\sigma^t Q  \subset B_{k,N}$) for every $t\in [r,s)$  and $(\sigma^{r-1} Q\cup \sigma^{s}Q)\subset G_{k,N}$.

An excursion $[r,s)$  of $Q$  into $B_{k,N}$ is said to \emph{enter} $\gamma_{>kN}$ if there exists $i \in [r,s)$ such that $\sigma^i Q \subset \gamma_{>kN}$.

\end{definition}

The next three lemmas are preparation for the proof of Proposition \ref{prop:goodcover}. These give us control on the return times to $K(k,N)$ and the length of excursions into $B_{k,N}$

\begin{lemma}\label{lem:A}
 If $[r,r+s)$ is an excursion of $Q$ into $B_{k,N}$ that does not enter $\gamma_{>kN}$ then $s<kN$. 
\end{lemma}

\begin{proof} Since the excursion does not enter  $\gamma_{>kN}$ we have that  $\sigma^rQ\subset \gamma_{N,kN}$.  Fix $x\in \sigma^{r}Q$.  By the definition of $\gamma_{N,kN}$ there exists  $x_0\in \alpha_{N,kN}$ such that $d(x,x_0)\le 2^{-(2(k+1)N+2)}$. Since $x_0 \in \alpha_{N,kN}$ we have that $n(x_0)\le kN$ and therefore $n_2(x_0)< kN$. In particular  $\sigma^t (x_0)\in \alpha_{\le N}$, for some $t\in [0,kN)$. Observe that $$d(\sigma^t (x),\sigma^t (x_0))\le  2^{-(2(k+1)N+2)+t}\le 2^{-((k+1)N+2)}.$$ This readily implies that $\sigma^t(x)\in U_{(k+1)N+2}(\alpha_{\le N})\subset G_{k,N}$. We conclude that $\sigma^{r+t} Q\subset G_{k,N}$, and therefore $s<kN$. 
\end{proof}

\begin{lemma}\label{lem:B}
 If  $Q\subset G_{k,N}$ then there exists $t \in [0, N)$ such that   $\sigma^tQ\subset K(k,N)$.
\end{lemma}

\begin{proof}
If $Q\subset K(k,N)$ there is nothing to prove. Assume that $Q\subset \gamma_{\le N}$. Let  $x\in Q$ and $y \in \alpha_{\le N}$ such that $d(x,y)\le 2^{-((k+1)N+2)}$. 
Since $y \in \alpha_{\le N}$ we have that $\sigma^t(y)\in K$, for some $t< N$. Observe that
\begin{equation*}
d(\sigma^t (x), \sigma^t (y))\le 2^{-((k+1)N+2)}2^{t}<2^{-(kN+2)}.
\end{equation*}
We conclude that there exists $t \in [0, N)$ such that $\sigma^t (x)\in U_{kN+2}(K)=K(k,N)$. This implies that  for some $t<N$ we have $\sigma^tQ\subset K(k,N)$.
\end{proof}

\begin{lemma}\label{lem:C}  If $[r,r+s)$ is an excursion of $Q$ into $\gamma_{>kN}$ such that $s\ge N$ then $\sigma^{r-1}Q\subset K(k,N)$.  
\end{lemma}

\begin{proof}   From the definition of an excursion, the set  $Q_0:=\sigma^{r-1}Q$ must lie in $G_{k, N}$, so to derive a contradiction we will assume that  $Q_0\subset \gamma_{\le N}$. Let  $x\in Q_0$. By the construction of $\gamma_{\le N}$  there exists $y \in \alpha_{\le N}$ such that $d(x,y)\le 2^{-((k+1)N+2)}$. Since $y\in \alpha_{\le N}$ there exists $t\le N$  such that $\sigma^t(y)\in K$. Therefore 
\begin{equation*}
d(\sigma^t(x),\sigma^t (y))\le 2^{-((k+1)N+2)+t}\le 2^{-(kN+2)}.
\end{equation*}
We conclude that $\sigma^t(x)\in U_{kN+2}(K)=K(k,N)$. This contradicts the fact the length of the excursion is larger than $N$. 
\end{proof}

\begin{definition}
Denote by  $m_{n,k,N}(Q)$ the number of  excursions of length  greater or equal to $kN$ into $B_{k,N}$ that enter $\gamma_{>kN}$  and let 
\begin{equation*}
E_{n,k,N} :=\# \left\{i \in[0,n): \sigma^i Q\subset B_{k,N} \right\}.
\end{equation*}
\end{definition}

The following result shows that  an atom $Q\in \beta_{k,N}^n$ such that $Q\subset K(k,N)\cap \sigma^{-(n-1)}K(k,N)$ can be covered by cylinders of length $n$ in a controlled way. This is an estimate closely related to \cite[Lemma 7.4]{ekp} (see also \cite[Proposition 4.5]{ve}). The constant $\hat\delta_\infty(K)$ naturally appears when we try to control the time spent in the `bad' part $B_{k,N}$. 

\begin{proposition} \label{prop:goodcover}   Let $\beta_{k,N}$  be the partition defined in  \eqref{def:beta}. Then an atom $Q\in \beta_{k,N}^n$ such that $Q\subset K(k,N)\cap \sigma^{-(n-1)}K(k,N)$, can be covered by at most 
\begin{equation*}
e^{E_{n,k,N}(Q)(\hat\delta_\infty(K)+\epsilon)}e^{m_{n,k,N}(Q)N(\hat\delta_\infty(K)+\epsilon)}
\end{equation*}
cylinders of length $n$.
\end{proposition}

\begin{proof}
To simplify notation we  drop the sub-indices $N$ and $k$. The proof of Proposition \ref{prop:goodcover} is by induction on $n$.  First decompose $[0,n-1]$ into
\begin{equation*}
 [0,n-1]=W_1\cup V_1\cup W_2\cup\cdots \cup V_s\cup W_{s+1},
 \end{equation*}
according to the  excursions into $B_{k,N}$ that contain at least one excursion into $\gamma_{>kN}$. More precisely, let $V_i=[m_i,m_i+h_i)$ and $W_i=[l_i,l_i+L_i)$ with $l_i+L_i=m_i$ and $m_i+h_i=l_{i+1}$.  The segment $V_i$ denotes an  excursion into $B_{k,N}$ that contains  an excursion into $\gamma_{>kN}$. Given $i\in \N$ define $J_i:= \sum_{j=1}^i |V_j|1_{[kN,\infty)}(|V_j|),$
where $1_{[kN,\infty)}$ is the characteristic function of the interval $[kN,\infty)$. Similarly define $H_i:= \sum_{j=1}^i 1_{[kN,\infty)}(|V_j|).$ Observe that $Q\subset K(k,N)$ implies that $Q$ is already contained in a cylinder of length $kN+2$.  \\

\emph{Step 1:} Assume that  $Q$ has been covered with $c_i$ cylinders of length $l_i$, where 
\begin{equation*}
c_i \le e^{J_i \left(\hat\delta_\infty(K)+\epsilon \right)}e^{NH_i \left(\hat\delta_\infty(K)+\e \right)}.
\end{equation*}

 (As mentioned above, the set  $Q$ is covered by one cylinder of length $1$, therefore take $c_1=1$.) We claim  that the same number  of cylinders of length $(l_i+L_i)$ cover $Q$. 
Observe that by hypothesis   $\sigma^{l_i}Q$ is contained in an element of $\beta'$, therefore $\diam(\sigma^{l_i}Q)\le 2^{-(kN+2)}$. Since the elements of $\beta'$ all have diameter smaller than $2^{-(kN+2)}$, the same holds if $Q$ spends some extra time in $\beta'$. By Lemma \ref{lem:A}, if $Q$ has an excursion into $B_{k,N}$ that does not enter $\gamma_{>kN}$, then it must come back to $\beta'$ before $kN$ iterates. In particular if the excursion into $B_{k,N}$ is $[p_i,p_i+q_i)$, then $q_i< kN$. Observe that  $\diam(\sigma^{p_i-1}Q)\le 2^{-(kN+2)},$ implies that $\diam(\sigma^{p_i+t}Q)\le 2^{-2},$ for every $t\in [0,kN)$. In particular the same holds for $t\in [0,q_i]$. Repeating this process we conclude that $\diam(\sigma^{t}Q)\le 2^{-2}$, for every $t\in [l_i,l_i+L_i)$. This immediately implies that $\sigma^{l_i}Q$ is contained in a cylinder of length $L_i$, which implies our claim. We go next to Step 2. 
\\

\emph{Step 2:}  Assume we have covered $Q$ with $c_i$ cylinders of length $m_i$, where 
\begin{equation*}
c_i\le  e^{J_i(\hat\delta_\infty(K)+\epsilon)}e^{NH_i(\hat\delta_\infty(K)+\e)}.
\end{equation*}
We want to estimate the number of cylinders of length $(m_i+h_i)$ needed to cover $Q$. Define $Q_i:=\sigma^{m_i-1}Q$. If we are able to cover $Q_i$ with $R$ cylinders of length $(h_i+1)$, then we will be able to cover $Q$ with $Rc_i$ cylinders of length $(m_i+h_i)$. We will separate into two cases:\\
\emph{Case 1}: $h_i< kN$.\\
Observe that $Q_i\subset G_{k,N}$   and is therefore contained in an element of $\beta'$, which implies $\diam(Q_i)\le 2^{-(kN+2)}$. This implies that $Q_i$ is contained in a cylinder of length $(kN+2)$. Since $h_i<kN$, this implies that $Q_i$ can be covered with one cylinder of length $(h_i+1)$.  We conclude that 
\begin{equation*}
c_{i+1}= c_i \le  e^{J_i(\hat\delta_\infty(K)+\epsilon)}e^{NH_i(\hat\delta_\infty(K)+\e)}=  e^{J_{i+1}(\hat\delta_\infty(K)+\epsilon)}e^{NH_{i+1}(\hat\delta_\infty(K)+\e)}.
\end{equation*}
\emph{Case 2}: $h_i\ge kN$.\\
By Lemma \ref{lem:C},  $Q_i=\sigma^{m_i-1}Q\subset K(k,N)$.  Observe that by assumption  $\sigma^{h_i+1}Q_i\subset \gamma_{\le N}$. By Lemma \ref{lem:B}   there exists $0\le t_i<N,$ such that $\sigma^{h_i+1+t_i}Q_i\subset K(k,N)$ (we assume $t_i$ is the smallest such number). 
We conclude that every $x\in Q_i$ satisfies $x\in K_{a_0}$, $\sigma^{h_i+1+t_i}(x)\in K_{a_0}$, and $\sigma^s x\in K^c$, for every $s\in \{1,\ldots, h_i+t_i\}$. In other words $Q_i\subset T_{h_i+t_i}(K)$. We now claim that $Q_i\subset \widehat{T}_{h_i+t_i}(K)$. Observe that if $x\in Q_i$, then $\sigma^{h_i+t_i+1}(x)\in K(k,N),$ and $\sigma^k(x)\in K(k,N)^c$, for every $k\in \{1,\ldots, h_i+t_i\}$. We argue by contradiction and suppose that $i(\sigma^k(x))>(h_i+t_i)-k$  for some $k\in \{1,\ldots, h_i+t_i\}$. This implies that $(\sigma^k(x))_j\le a_j$, for $j\in \{0,\ldots, h_i+t_i-k\}$. Observe that $(\sigma^k(x))_{h_i+t_i-k+j+1}=(\sigma^{h_i+t_i+1}(x))_{j}$, and  for $j\in\{0,\ldots,kN+1\}$ we have $(\sigma^{h_i+t_i+1}(x))_{j}\le a_j$. We conclude that $(\sigma^k(x))_{h_i+t_i-k+j+1}\le a_j$, for $j\in \{0,\ldots,kN\}$. In particular we have that $(\sigma^k(x))_j\le a_j$, for every $j\in \{0,\ldots,kN+1\}$, which contradicts that $\sigma^k(x)\in K(k,N)^c$, completing the proof of our claim. 
This implies,  from the definition of $\hat\delta_\infty(K)$, that $Q_i$ can be covered by at most $e^{(h_i+t_i)(\hat\delta_\infty(K)+\e)}$ cylinders of length $(h_i+1+t_i)$; and by at most $e^{(h_i+N)(\hat\delta_\infty(K)+\e)}$ cylinders of length $(h_i+1)$. We conclude that $Q$ can be covered by at most $c_{i+1}$ cylinders of length $(n_i+h_i)$, where 
\begin{align*}c_{i+1}\le & e^{(h_i+N)(\hat\delta_\infty(K)+\e)}(e^{J_i(\hat\delta_\infty(K)+\epsilon)}e^{NH_i(\hat\delta_\infty(K)+\e)})\\
&= e^{J_{i+1}(\hat\delta_\infty(K)+\epsilon)}e^{NH_{i+1}(\hat\delta_\infty(K)+\e)}.
\end{align*}
 Adding these steps together and noting that $J_s= E_{n,k,N}(Q)$ and $H_s= m_{n,k,N}(Q)$ completes the proof of the proposition.
\end{proof}

The idea now is to use Proposition \ref{prop:goodcover} to compare the entropy of a measure with the corresponding entropy of our partition $\beta_{k,N}$. This is a natural idea: the map $\mu\mapsto h_\mu(\beta_{k,N})$ is typically better behaved under sequences of measures; at this point we crucially use that the partition $\beta_{k,N}$ is finite. 

\begin{proposition}\label{prop:ineq} Let $\beta_{k,N}$ be the partition defined in \eqref{def:beta} and  $\mu$ an ergodic $\sigma$-invariant probability measure satisfying $\mu(K(k,N))>0$. Then
\begin{equation*}
h_{\mu}(\sigma)\le h_{\mu}(\beta_{k,N})+\left( \mu(B_{k,N}) + \frac{1}{k} \right)      (\hat\delta_\infty(K)+\epsilon).
\end{equation*}
\end{proposition}

\begin{proof}
 To simplify notation we denote  the partition $\beta_{k,N}$ by  $\beta$. We will apply Theorem \ref{katformula}, so the main task is to estimate $N_\mu(n,1,\delta)$ for some $\delta\in (0, 1)$.
Since $\mu$ is an ergodic measure such that $\mu(K(k,N))>0$ there exists $\delta_1>0$ and an increasing  sequence $(n_i)_{i}$ satisfying
\begin{equation*}
\mu(K(k,N)\cap \sigma^{-n_i}K(k,N)) > \delta_1,
\end{equation*}
for every $i\in \N$. Given $\e_1>0$, by the Shannon-McMillan-Breiman theorem  the set  
\begin{equation*}
\D_{\e_1,N}= \left\{x\in X : \forall n\geq N, \mu(\beta^n(x))\geq \exp(-n(h_\mu(\beta)+\e_1)) \right\},
\end{equation*}
satisfies
\begin{equation*}
\lim_{N \to \infty} \mu \left(	\D_{\e_1,N}\right) =1. 
 \end{equation*}
By Birkhoff's Ergodic Theorem there exists a set $W_{\e_1} \subset \Sigma$  satisfying $\mu(W_{\e_1})>1-\frac{\delta_1}{4}$ and $n(\e_1) \in \N$ such that
for every $x\in W_{\e_1}$ and $n\ge n(\e_1)$,
\begin{equation*}
\frac{1}{n}\sum_{i=0}^{n-1} 1_{B_{k,N}}(\sigma^n x) < \mu(B_{k,N})+\e_1.
\end{equation*}
Define
\begin{equation*}
X_i:= W_{\e_1}\cap  \D_{\e_1,n_i}\cap K(k,N)\cap \sigma^{-n_i}K(k,N).
\end{equation*}
So for  sufficiently large values of $i \in \N$, by construction we have that  $\mu(X_i)>\frac{\delta_1}{2}$. In what follows we will assume that $i \in \N$ is large enough that it satisfies this condition.

By  definition of $\D_{\e_1,n_i}$ the set  $X_i$ can be covered by $\exp(n_i(h_\mu(\beta)+\e_1))$ many elements of $\beta^{n_i}$. 
We will make use of Proposition \ref{prop:goodcover} to efficiently  cover each of those atoms by cylinders.  Let $Q\in \beta^{n_i}$ be an atom intersecting $X_i$. In particular $Q\in K\cap\sigma^{-(n-1)}K$.
It follows from the definition of $W_{\e_1}$ that
\begin{equation*}
E_{n_i,k,N}(Q) <\left(\mu(B_{k,N})+\e_1 \right)n_i.
\end{equation*}
Moreover,
\begin{equation*}
m_{n_i,k,N}(Q)\le \frac{1}{kN}n_i. 
\end{equation*}
Indeed, each of the excursions counted in $m_{n_i,k,N}$ has length at least $kN$, which implies that the number of excursions can not be larger than $\frac{1}{kN}n_i$.  
Therefore Proposition \ref{prop:goodcover} implies that 
\begin{equation*}
N_\mu\left(n_i, 1,1-\frac{\delta_1}{2} \right) \leq
e^{n_i(h_\mu(\beta)+\e_1)}e^{n_i(\hat\delta_\infty(K)+\epsilon)(\mu(B_{k,N})+\e_1)}e^{\frac{1}{kN}n_iN(\hat\delta_\infty(K)+\e)}.
\end{equation*}
It now follows from Katok's entropy formula (see Theorem \ref{katformula}) that
\begin{equation*}
h_\mu(\sigma) \le  h_\mu(\beta_{k,N})+\e_1+(\hat\delta_\infty(K)+\e)(\mu(B_{k,N})+\e_1)+\frac{1}{k}(\hat\delta_\infty(K)+\e).
\end{equation*}
Since $\epsilon_1>0$ was arbitrary the proof is complete.  
\end{proof}

As in Proposition \ref{prop:ineq} we denote the partition $\beta_{k,N}$ by  $\beta$. We may assume, possibly after refining the partition,  that   
$$\beta=\{C^1,\ldots,C^q,R\},$$ where each $C^i$ is a cylinder for the original partition and $R=\gamma_{>kN}$ is the complement of a finite collection of cylinders. 
For simplicity we still denote this partition by  $\beta$.
We emphasise that Proposition \ref{prop:ineq} still holds for this new partition.

Define, for large $m$,  $F_m:=\bigcap_{i=0}^{m-1}\sigma^{-i} R$. 
We will require the following continuity result.  

\begin{proposition}\label{prop:atom} Suppose that $(\mu_n)_n$ is a sequence of ergodic probability measures converging on cylinders to an invariant measure $\mu$, where $\mu(\Sigma)>0$.  For every $P\in \beta^m\setminus \{F_m\},$ we have  
\begin{equation*}
\lim_{n\to\infty}\mu_n(P)=\mu(P).
\end{equation*}
\end{proposition}

\begin{proof} In order to prove the proposition we will need the following fact. 

\begin{claim}\label{claim:preatom} Let  $(H_i)_i$ be a collection of cylinders and $(p_i)_i$ a sequence of natural numbers. Then $H_0\cap \sigma^{-p_1}H_1\cap \cdots\cap \sigma^{-p_k}H_k$, is either a finite collection of cylinders,  or the empty set. 
 \end{claim}

\begin{proof} We begin with the case $k=2$, in other words, we will prove that  if $C$ and $D$ are cylinders, then for every $p\in \N$ the set $C\cap \sigma^{-p}D$ is a finite collection of cylinders or the empty set.  If the length of $C$ is larger than or equal to $p$ then $C\cap \sigma^{-p}D$ is empty or a cylinder. If $p$ is larger than the length of $C$, then we use that there are only finitely many admissible words of given length connecting two fixed symbols. More precisely, if $C=[x_0,\ldots,x_{h-1}]$ and $D=[y_0,\ldots,y_{t-1}]$, then there are finitely many admissible words of length $p-h+2$ connecting $x_{h-1}$ and $y_0$. We conclude that $C\cap \sigma^{-p}D$ is a finite collection of cylinder  or the empty set. The same argument gives us the proof of the claim for arbitrary $k$. 
\end{proof}

Let $P=S_0\cap \sigma^{-1}S_1\cap\cdots\cap\sigma^{-(m-1)}S_{m-1}$, where $S_i\in \beta$ and  $P_k:=\bigcap_{i=k}^{m-1}\sigma^{-(i-k)}S_i$. Define $B=B(P):=\{i\in\{0,\ldots,m-1\}: S_i=R\}$, $G=G(P):=\{0,\ldots,m-1\}\setminus B$, and $k=k(P):=(\min G)-1$. By assumption we know that $G\ne \emptyset$. Let $Q_0=Q_0(P):=\bigcup_{i=0}^k \sigma^{-i}R$, $Q_1=Q_1(P):=\bigcap_{i\in G} \sigma^{-i}S_i$, and $Q_2=Q_2(P):=\bigcap_{i\in B\cap (k,\infty)}\sigma^{-i}S_i$. We will first consider the case $k=-1$, where $Q_0=\emptyset$. 

\begin{claim}\label{claim:-1} Let $P=\bigcap_{i=0}^{m-1}\sigma^{-i}S_i$, where $S_0\in \{C^1,\ldots,C^q\}$. Then $$\lim_{n\to\infty} \mu_n(P)=\mu(P).$$\end{claim}

\begin{proof}
Since $Q_1$ is the disjoint union of $P=(Q_1\cap Q_2)$ and $(Q_1\cap Q_2^c)$, for every $n\in \N$ we obtain that 
\begin{equation*}
\mu_n(P)=\mu_n(Q_1)-\mu_n(Q_1\cap Q_2^c).
\end{equation*}
Observe that  
\begin{align*} Q_1\cap Q_2^c&= \left(\bigcap_{j\in G}\sigma^{-j}S_j \right) \cap \left(\bigcup_{i\in B} \sigma^{-i}R^c \right)=\bigcup_{i\in B} \left(\sigma^{-i}R^c\cap \bigcap_{j\in G}\sigma^{-j}S_j \right).
\end{align*}
From  Claim \ref{claim:preatom} we conclude that for every $i\in B$ the sets $Q_1$ and $(\sigma^{-i}R^c\cap \bigcap_{j\in G}\sigma^{-j}S_j)$ are a finite union of cylinders or the empty set. Therefore,  $Q_1$ and $Q_1\cap Q_2^c$ are a finite union of cylinders, or the empty set. From this we immediately obtain that 
\begin{align*}\lim_{n\to\infty}\mu_n(P)&=\lim_{n\to\infty}\mu_n(Q_1)-\lim_{n\to\infty}\mu_n(Q_1\cap Q_2^c)=\mu(Q_1)-\mu(Q_1\cap Q_2^c)=\mu(P),\end{align*}
which proves the claim.
\end{proof}

We now explain how to reduce the case $k\ge 0$ to Claim \ref{claim:-1}. Observe that $P=R\cap \sigma^{-1}P_1$, therefore $\sigma^{-1}P_1$ is the disjoint union between $P$ and $S_1:=(R^c\cap\sigma^{-1}P_1)=\bigcup_{i=1}^q(C^i\cap \sigma^{-1}P_1)$. Thus,
\begin{align*} \mu_n(P)&=\mu_n(\sigma^{-1}P_1)- \mu_n(R^c\cap\sigma^{-1}P_1)=\mu_n(P_1)-\sum_{i=1}^q \mu_n(C^i\cap\sigma^{-1}P_1).
\end{align*}
By Claim \ref{claim:-1} we know that $\lim_{n\to\infty}\mu_n(C^i\cap\sigma^{-1}P_1)=\mu(C^i\cap\sigma^{-1}P_1)$. Therefore it suffices to prove that $\lim_{n\to\infty}\mu_n(P_1)=\mu(P_1)$.  Applying the above argument $k$ times we obtain that  the original problem is reduced to $\lim_{n\to\infty}\mu_n(P_{k+1})=\mu(P_{k+1})$
Since $P_{k+1}=S_{k+1}\cap \sigma^{-1}P_{k+2}$, where $S_{k+1}\in \{C^1,\ldots,C^q\}$,  we conclude the proof of the proposition by applying Claim \ref{claim:-1}. 
\end{proof}

\begin{proof}[Proof of Theorem~\ref{pre}] 
We first consider the case in which not all the mass escapes, that is, we assume that $\mu(\Sigma)>0$. Let $\varepsilon_0>0$.
Choose $m \in \N$ sufficiently large such that 
\begin{equation*}
h_{\frac{\mu}{|\mu|}}(\sigma)+\varepsilon_0>\frac{1}{m}H_{\frac{\mu}{|\mu|}}(\beta^{m})\quad , \quad 2\frac{e^{-1}}{m}<\frac{\varepsilon_0}{2} \quad and \quad -\left(\frac1m\right)\log |\mu|<\varepsilon_0.
\end{equation*}
Then 
\begin{equation*}
 h_{\frac{\mu}{|\mu|}}(\sigma)+\varepsilon_0>
 \frac{1}{|\mu|}\frac1m \left(\log|\mu|- \sum_{P\in \beta^m} \mu(P)\log\mu(P)\right)
\end{equation*}
and hence
\begin{equation*}
|\mu| h_{\frac{\mu}{|\mu|}}(\sigma)+2\varepsilon_0>-\frac{1}{m}\sum_{P\in \beta^m} \mu(P)\log\mu(P).
\end{equation*}
It follows from Proposition \ref{prop:atom}  that
\begin{equation*} 
\lim_{n\to\infty} \sum_{Q\in \beta^{m}\setminus\{F_m\}}\mu_n(Q)\log\mu_n(Q) = \sum_{Q\in \beta^{m}\setminus\{F_m\}}\mu(Q)\log\mu(Q).
\end{equation*}
For sufficiently large $n \in \N$  we have the inequality 
\begin{equation*}
|\mu|h_{\frac{\mu}{|\mu|}}(\sigma)+3\varepsilon_0\ge \frac{1}{m}H_{\mu_n}(\beta^m).
\end{equation*}
By Proposition \ref{prop:ineq}, we have that 
\begin{align*}
|\mu| h_{\frac{\mu}{|\mu|}}(\sigma)+3\varepsilon_0 \geq & \frac{1}{m}H_{\mu_n}(\beta^{m})\geq h_{\mu_n}(\sigma,\beta)\\
\ge & h_{\mu_n}(\sigma)-(\hat\delta_\infty(K)+\e)\mu_n(B_{k,N})-\frac{1}{k}(\hat\delta_\infty(K)+\e).
\end{align*}
Since $\varepsilon_0>0$ is arbitrary we get
\begin{align}\label{forA} \limsup_{n\to\infty} h_{\mu_n}(\sigma)\leq  |\mu| h_{\frac{\mu}{|\mu|}}(\sigma)+(\hat\delta_\infty(K)+\e)(1-\mu(G_{k,N})) +\frac{1}{k}(\hat\delta_\infty(K)+\e).
\end{align}
We stress  that Proposition \ref{prop:ineq} can be applied  for arbitrary $k, N \in \N$ since $K\subset \supp(\mu_n)$ and therefore $\mu_n(K(k,N))>0$.  
Finally, letting  $k\to \infty$ and $\e\to 0$ we obtain the inequality
\begin{equation*}
\limsup_{n\to \infty} h_{\mu_n}(\sigma)\le |\mu|h_{\mu/|\mu|}(\sigma)+\left(1-\sup_{k,N}\mu(G_{k,N}) \right)\hat\delta_\infty(K).
\end{equation*}
Observe that $Y\subset \bigcup_{k,N} G_{k,N}$, therefore $\mu(Y)\le \mu \left(\bigcup_{k,N} G_{k,N} \right)=\sup_{k,N}\mu(G_{k,N})$.
We conclude that 
\begin{equation*}
\limsup_{n\to \infty} h_{\mu_n}(\sigma)\le |\mu|h_{\mu/|\mu|}(\sigma)+(1-\mu(Y))\hat\delta_\infty(K).
\end{equation*}
The case  $\mu(\Sigma)=0$ follows directly from Proposition \ref{prop:ineq} since $h_{\mu_n}(\sigma,\beta)\to 0$ and $\mu_n(B_{k,N})\to 1$ as $n\to \infty$.
\end{proof}

\section{Proof of  Theorem \ref{thm:main}}\label{finalproof}
In this section we prove our main result. We start with a simple result we will need later.

\begin{lemma}\label{lem:Kineq} 
Let $(a_j)_j$ and $(b_j)_j$ be sequences of natural numbers such that for every $i\in\N_0$ we have  $a_0=b_0$ and $a_j\le b_j$. Then $\hat\delta_\infty(K((b_j)_j))\le \hat\delta_\infty(K((a_j)_j)).$
\end{lemma}

\begin{proof} Denote by $K_1:=K((a_j)_j)$ and $K_2:=K((b_j)_j)$. Recall that associated to each compact set defined in this way there is a function $i$ (see \eqref{def:i} for the definition).  Denote the function $i$ associated to $K_1$ (resp. $K_2$)  by $i_1$ (resp. $i_2$). It follows from the hypothesis that $K_1\subset K_2$. In particular we have that $K_2^c\subset K_1^c$ and therefore $T_n(K_2)\subset T_n(K_1)$ (see \eqref{def:T} for the definition of $T$). Moreover, we have that
\begin{equation*}
\widehat{T}_n(K_2)\subset \widehat{T}_n(K_1),
\end{equation*}
(see \eqref{def:That} for the definition of $\widehat{T}$). Indeed, let  $x\in \widehat{T}_n(K_2)$, we have that $i_2(\sigma^k(x))\le n-k$. In particular 
\begin{equation*}
(\sigma^k(x))_{i_2(\sigma^k(x))}>b_{i_2(\sigma^k(x))}\ge a_{i_2(\sigma^k(x))}.
\end{equation*}
We conclude that $i_1(\sigma^k(x))\le i_2(\sigma^k(x))\le n-k$, and therefore $x\in \widehat{T}_n(K_1)$. Thus $\widehat{T}_n(K_2)\subset \widehat{T}_n(K_1)$, which readily implies that for every $n \in \N$ we have $\hat z_n(K_2)\le \hat z_n(K_1)$. 
\end{proof}

In the next lemma we establish a relation  between the quantities $\hat\delta_\infty(K)$ and $\delta_\infty(q)$, which in turn is necessary to relate Theorem \ref{pre} with Theorem \ref{thm:main}.  As mentioned before, in the definition of $\hat\delta_\infty(K)$ we covered the sets $\widehat{T}_n(K)$ (and not $T_n(K)$ which may seem more natural) in order to ensure this result.

\begin{lemma}\label{lem:Kineq2} Let  $(\Sigma, \sigma)$ be a CMS satisfying the $\F-$property, and $M , q  \in  \N$. Then  there exists a sequence of natural numbers $(a_i)_i$ such that $a_0=q$, and 
\begin{equation*}
\hat\delta_\infty(K)\le \delta_\infty(M,q),
\end{equation*}
where $K=K((a_i)_i)$
\end{lemma}

\begin{proof} 
 Let $i \in \N$. Since $\Sigma$ satisfies the $\F-$property,  there are finitely many cylinders of the form $[x_0,\ldots,x_n]$, where $x_0\le q$, $x_n\le q$, and $n\le iM$. Thus, only a finite collection of symbols from the alphabet are used in this collection of cylinders. Denote by $r_i \in \N$ the largest of this collection of symbols. Inductively define  $(a_i)_i  \subset \N$ so that:
\begin{equation*} 
 a_{i+1} >a_i  \quad \text{ and  } \quad a_i > r_i.
\end{equation*}
We now prove that the set  $K=K((a_i)_i)$ is such that $\hat z_n(K)\le \delta_\infty(M,q)(n)$, for every $n\in \N$. Recall that $\hat z_n(K)$ is the minimal number of cylinders of length $(n+2)$ needed to cover  $\widehat{T}_n$. Let $x=(x_0, x_1, \dots) \in \widehat{T}_n$,
\begin{equation*}
E:= \left\{k\in \{0,\ldots,n+1\}:x_k\le a_0 \right\},
\end{equation*}
and $B:=\{0,\ldots,n+1\}\setminus E$. For $k\in E$ we define $p_k:=i(\sigma^k(x))$. We emphasise that since $k\in E$ then $x_k\le a_0$, thus $p_k=i(\sigma^k(x))\ge 1$. Let  $r \in E$ and observe that $x_{p_r+r}=(\sigma^r(x))_{p_r}>a_{p_r}$, where $p_r\le n-r$. 
Because of the choice of $a_{p_r}$, there is no admissible word of length less or equal to $ p_rM$ connecting $x_{p_r+r}$ and a symbol in the set $\{0, 1 \dots, q\}$. Since $x_{n+1}\le q$, this means that we must have $p_r+r+(p_rM)\le n+1$. Moreover, for every $0\le  m< p_rM$
we have that $p_r+r+m\in B$. In other words, the interval $[r,r+p_r+p_rM)$ has at least $p_rM$ elements in $B$, equivalently, at most $p_r$ elements in $E$.  Since this argument holds for every $r\in E$ we conclude that $M\#E\le n+2$ and therefore 
\begin{equation} \label{eq:ele}
  \#E\le \frac{n+2}{M}.
\end{equation}   
From \eqref{eq:ele} it follows that every $x\in \widehat{T}_n$ belongs to a cylinder of the form $[x_0,\ldots,x_{n+1}]$, where $x_0\le q$, $x_{n+1}\le q$ and
\begin{equation*}
 \#\{i\in\{0,1,\ldots,n+1\}: x_i\le q\}\le \frac{n+2}{M}.
 \end{equation*}
This implies that $\hat z_n(K)\le z_n(M,q)$, for every $n\in \N$. Therefore $\hat\delta_\infty(K)\le \delta_\infty(M,q)$. 
\end{proof}
Define $\hat{\delta}_\infty(q):=\inf_{(a_i)_i:a_0=q} \hat\delta_\infty(K((a_i)_i)$.

\begin{corollary}\label{cor:deltaineq} 
For every $q\in \N$ we have  $\hat{\delta}_\infty(q)\le \delta_\infty(q)$.
\end{corollary}

\begin{proof} Combine Lemma \ref{lem:Kineq} with Lemma \ref{lem:Kineq2}.
\end{proof}

We now  prove Theorem \ref{thm:main}.  

\begin{proof}[Proof of  Theorem \ref{thm:main}] 
Let $(a_i)_i$ be a sequence of non-negative integers and $K:=K((a_i)_i)$ the corresponding compact set. We assume $K$ large enough so that 
there exists a periodic measure  $\mu_p$  with $\mu_p(K)>0$.  We will prove that 
\begin{align}\label{11} 
\limsup_{n\to \infty} h_{\mu_n}(\sigma)\le |\mu|h_{\mu/|\mu|}(\sigma)+(1-\mu(K))\hat\delta_\infty(K).
\end{align}

Let  $\mu'_{n}:=(1-\frac{1}{n})\mu_n+\frac{1}{n}\mu_p$. Observe that for every $n\in\N$ we have $\mu_n'(K)>0$. It follows from Proposition \ref{dense} that there exists an ergodic measure $\nu_n$ arbitrarily close in the weak$^*$ topology to $\mu_{n}'$  such that $h_{\nu_n}(\sigma)>h_{\mu'_n}(\sigma)-\frac{1}{n}$.  In particular, we can assume that $\nu_n(K(n,n))>0$ and that $(\nu_n)_n$ converges on cylinders to $\mu$.

Let  $k , N \in \N$. If $n>\max\{k,N\}$ then $K(n,n)\subset K(k,N)$, therefore $\nu_n(K(k,N))>0$.  It now follows from \eqref{forA} that
\begin{equation*}  
\limsup_{n\to\infty} h_{\nu_n}(\sigma)\leq  |\mu| h_{\frac{\mu}{|\mu|}}(\sigma)+(\hat\delta_\infty(K)+\e)(1-\mu(G_{k,N})) +\frac{1}{k}(\hat\delta_\infty(K)+\e).
\end{equation*}
Letting $k$ tend to infinity and $\e$ to zero we obtain 
\begin{equation*}
\limsup_{n\to\infty} h_{\nu_n}(\sigma)\leq  |\mu| h_{\frac{\mu}{|\mu|}}(\sigma)+(1-\mu(K))\hat\delta_\infty(K).
\end{equation*}
Since  $h_{\nu_n}(\sigma)>h_{\mu'_n}(\sigma)-\frac{1}{n}=(1-\frac{1}{n})h_{\mu_n}-\frac{1}{n}$, then
\begin{equation*}
 \limsup_{n\to \infty} h_{\mu_n}(\sigma)\le \limsup_{n\to \infty} h_{\nu_n}(\sigma),
 \end{equation*}
from which \eqref{11} follows. 

The argument above also holds for every set $K'=K((b_i)_i)$, where $a_0=b_0$ and $a_i\le b_i$. Observe that  $\sup_{(b_i)_i:b_0=a_0} \mu(K((b_i)_i))=\mu(K_{a_0})$. Thus, it is a consequence of  Corollary \ref{cor:deltaineq} that
\begin{align*} \limsup_{n\to \infty} h_{\mu_n}(\sigma)&\le |\mu|h_{\mu/|\mu|}(\sigma)+(1-\mu(K_{a_0}))\hat{\delta}_\infty(a_0)\\
& \le |\mu|h_{\mu/|\mu|}(\sigma)+(1-\mu(K_{a_0}))\delta_\infty(a_0).
\end{align*}
Letting $a_0$ tend to infinity concludes the proof of Theorem \ref{thm:main}.
\end{proof}

\section{Variational principle for the entropies at infinity}\label{entinf1}

 In this section we prove  Theorem \ref{thm:vpinf}. That is, we prove a variational principle at infinity: the measure theoretic entropy at infinity coincides with its topological counterpart.  
 
For each pair $(i,j)\in\N^2$  choose a non-empty cylinder $w(i,j)$ of length  $\ell(i,j)+1$ such that
\begin{equation*}
w(i,j):=[i, \dots, j]= [(w(i,j)_0, \dots , w(i,j)_{\ell(i,j)}].
\end{equation*}
Let $\phi: \Sigma \to \R$ be a potential and define
\begin{equation*}
Z_n(\phi, a,b):=\sum_{x:\sigma^{n+\ell(b,a)}(x)=x}\exp \left(S_{n+\ell(b,a)}\phi(x) \right)1_{[a]\cap \sigma^{-n}w(b,a)}(x).
\end{equation*}
In the following lemma we show that the Gurevich pressure can be computed by means of the partition function $Z_n(\phi, a,b)$; this will be used in Lemma \ref{lem:testineq}. 

\begin{lemma}\label{lem:equivgur} Let $(\Sigma,\sigma)$ be a transitive CMS and  $\phi: \Sigma \to \R$ a bounded potential with summable variations. Then for every pair $(a,b) \in \N^2$ we have that
 \begin{equation*}
P_{G}(\phi)=\limsup_{n\to\infty}\frac{1}{n}\log Z_n(\phi,a,b).
\end{equation*}
\end{lemma}

\begin{proof} Let $C=\|\phi\|_0$ and $D=\sum_{k=2}^\infty \text{var}_k(\phi)$. It follows from the definition of $Z_n(\phi,a,b)$ that 
\begin{equation*}
Z_{n+\ell(b,a)}(\phi,a)=\sum_{x:\sigma^{n+\ell(b,a)}(x)=x}\exp(S_{n+\ell(b,a)}\phi(x))1_{[a]}(x)  \ge Z_n(\phi,a,b).
\end{equation*}
In particular we obtain that 
 \begin{equation*}
 P_{G}(\phi)=\limsup_{n\to\infty}\frac{1}{n}\log Z_n(\phi,a)\ge \limsup_{n\to\infty}\frac{1}{n}\log Z_n(\phi,a,b).
 \end{equation*}
Let  $\mathbb{P}_n:=w(a,b)\cap \sigma^{-n}w(b,a)$. Note that 
\begin{align*} 
Z_n(\phi,a,b)&\ge \sum_{x:\sigma^{n+\ell(b,a)}(x)=x}\exp \left(S_{n+\ell(b,a)}\phi(x) \right)1_{P_n}(x)\\
& \ge e^{-(\ell(a,b)+\ell(b,a))C}   \sum_{x:\sigma^{n+\ell(b,a)}(x)=x} \exp \left(S_{n -\ell(b,a)}\phi(\sigma^{\ell(a,b)}x) \right)1_{{P}_n}(x).\\
\end{align*}
Observe that if $x=(x_0, x_1, \dots) \in \mathbb{P}_n$, then $x_{\ell(b,a)}=x_{n}=b$. Define the periodic point $y(x):=\overline{x_{\ell(b,a)}\ldots x_{n-1}}$. The function $y$ establishes a one-to-one correspondence between points in $x\in \mathbb{P}_n$ such that $\sigma^{n+\ell(b,a)}(x)=x$, and periodic points  of length $n-\ell(b,a)$ in $[b]$. Moreover, note that if $x\in\mathbb{P}_n$, then 
\begin{equation*}
\left |S_{n-\ell(b,a)} \left(\phi(\sigma^{\ell(a,b)}x) \right)-S_{n-\ell(b,a)} \left(\phi(y(x)) \right) \right|\le D.
\end{equation*}
We conclude that 
\begin{eqnarray*} 
\sum_{x:\sigma^{n+\ell(b,a)}(x)=x} \exp \left(S_{n -\ell(b,a)}\phi(\sigma^{\ell(a,b)}x) \right)1_{\mathbb{P}_n}(x) \ge &\\
e^{-D} \sum_{x:\sigma^{n-\ell(b,a)}(y)=y} \exp \left(S_{n -\ell(b,a)}\phi(y) \right)1_{[b]}(x).
\end{eqnarray*}
That is $Z_n(\phi,a,b)\ge e^{-(\ell(a,b)+\ell(b,a))C-D}Z_{n-\ell(b,a)}(\phi,b)$ and therefore 
\begin{equation*}
\limsup_{n\to\infty}\frac{1}{n}\log Z_n(\phi,a,b)\ge P_{G}(\phi).
\end{equation*}
\end{proof}

\begin{remark} Note that in Lemma \ref{lem:equivgur} the assumption $\|\phi\|_0<\infty$ is too strong for what is required: it suffices  to  assume that 
for every $n\in \N$ we have $\sup_{x\in [n]}|\phi(x)|<\infty$.
\end{remark}

We say that a point $x\in\Sigma$ belongs to the set  $Per(q,M,n)$ if the following properties hold:
\begin{enumerate}
\item $\sigma^n(x)=x$.
\item If $x\in [x_0,\ldots,x_{n-1}]$, then $x_0\le q$, and $\#\{k\in\{0,\ldots,n-1\}:x_k\le q\}\le \frac{n}{M}$.
\end{enumerate}

The following lemma is important in our proof of Theorem \ref{thm:vpinf} as it will allow us to find a sequence of invariant probability measures which converges to the zero measure and entropies approach the topological entropy at infinity.

\begin{lemma}\label{lem:testineq} Let $\phi: \Sigma \to \R$ be a bounded  potential of summable variations such that 
\begin{equation*}
\lim_{n\to\infty}\sup_{x\in [n]} |\phi(x)|=0.  
\end{equation*}
Then $P_G(\phi)\ge \delta_\infty$.
\end{lemma}

\begin{proof}  
For every $\e>0$ there exists $N_0=N_0(\e) \in \N$ such that $\sup_{x \in [n]} |\phi(x)| \le \e$, for every $n\ge N_0$.  By Lemma \ref{lem:equivgur}, for sufficiently large values of $n \in \N$,  since $Z_n(a,n)(\phi) \leq \exp  (n P_G(\phi) + \e)$ there exists $N'=N'(N_0) \in \N$ such that
\begin{equation*}
N' \exp  \left(n P_G(\phi) + \e \right)  \geq \sum_{(a,b)\in \{1,\ldots,N_0\}^2} Z_n(\phi,a,b).
\end{equation*}
That is, 
\begin{equation*}
P_G(\phi)\ge \limsup_{n\to\infty} \frac{1}{n}\log\sum_{(a,b)\in \{1,\ldots,N_0\}^2} Z_n(\phi,a,b).
\end{equation*}
Define 
\begin{equation*}
\mathbb{T}_n(a,b):=\sum_{x\in Per(N_0,M,n+\ell(b,a))} \exp \left(S_{n+\ell(b,a)}\phi(x) \right)1_{[a]\cap \sigma^{-n}w(b,a)}(x),
\end{equation*}
and observe that $Z_n(\phi,a,b)\ge \mathbb{T}_n(a,b).$ Recall that $x\in Per(N_0,M,n+\ell(b,a))$ implies that 
\begin{equation*}
\# \left\{k\in\{0,\ldots,n+\ell(b,a)-1\}:x_k\le N_0 \right\} \le \frac{n+\ell(b,a)}{M}.
\end{equation*}
It follows from the choice of $N_0$ that 
\begin{equation*}
S_{n+\ell(b,a)}\phi(x)\ge  -(n+\ell(b,a))\e-\frac{\|\phi\|_0}{M}(n+\ell(b,a)).
\end{equation*}
In particular 
\begin{equation*}
\mathbb{T}_n(a,b)\ge \#  \left\{ Per(N_0,M,n+\ell(b,a))\cap [a]\cap\sigma^{-n}w(b,a) \right\} e^{-(n+\ell(b,a))(\e+\frac{\|\phi\|_0}{M})}.
\end{equation*}
Denote by $\mathcal{W}_n(a,b,N_0,M)$ the collection of cylinders of the form  $[x_0,\ldots,x_{n}]$, where $x_0=a$, $x_n=b$, and $\#\{k\in\{0,\ldots,n\}:x_k\le N_0\}\le \frac{n+1}{2M}$. 
In order to estimate the number of these using periodic points, to each cylinder $[x_0,\ldots,x_n]\in \mathcal{W}_n(a,b,N_0,M)$ we associate the cylinder
\begin{equation*}
D=[y_0,\ldots,y_{n+\ell(b,a)}]=[x_0,\ldots,x_n,(w(b,a))_1,\ldots,w(b,a)_{\ell(b,a)}].
\end{equation*}
 Observe that $y_0=a$, $y_{n+\ell(b,a)}=a$, and $$\#\{k\in \{0,\ldots,n+\ell(b,a)\}:y_k\le N_0\}\le \frac{n+1}{2M}+\ell(b,a).$$  For $n \in \N$ sufficiently large we can assume that $\frac{n+1}{2M}+\ell(b,a)\le \frac{n+\ell(b,a)}{M}$. In particular the periodic point associated to $D$ belongs to $Per(N_0,M,n+\ell(b,a))\cap[a]\cap\sigma^{-n}w(b,a)$. We conclude that
 \begin{equation*}
 \#\mathcal{W}_n(a,b,N_0,M)\le \#\left\{Per(N_0,M,n+\ell(b,a))\cap[a]\cap\sigma^{-n}w(b,a)\right\}.
\end{equation*} 
Observe that $\sum_{(a,b)\in \{1,\ldots,N_0\}^2} \#\mathcal{W}_n(a,b,N_0,M)=z_{n-1}(2M,N_0)$, which implies 
\begin{equation*}
z_{n-1}(2M,N_0)\le \sum_{(a,b)\in \{1,\ldots,N_0\}^2}  \#\left\{Per(N_0,M,n+\ell(b,a))\cap[a]\cap\sigma^{-n}w(b,a) \right\}.
\end{equation*}
Hence, writing $\ell_{N_0} := \max_{(a,b)\in \{1,\ldots,N_0\}^2}\ell(b,a))$,  we obtain that 
\begin{equation*}
\sum_{(a,b)\in \{1,\ldots,N_0\}^2}Z_n(\phi,a,b)\ge \sum_{(a,b)\in \{1,\ldots,N_0\}^2}\mathbb{T}_n(a,b)\ge z_{n-1}(2M,N_0) e^{-(n+\ell_{N_0})(\e+\frac{\|\phi\|_0}{M})},
\end{equation*}
and therefore 
\begin{equation*}
P_G(\phi)\ge \limsup_{n\to\infty}\frac{1}{n}\log \sum_{(a,b)\in \{1,\ldots,N_0\}^2}Z_n(\phi,a,b)\ge \delta_\infty(2M,N_0)-\e-\frac{\|\phi\|_0}{M}.
\end{equation*}
Letting $M\to\infty$ we obtain that $P(\phi)\ge \delta_\infty(N_0)-\e$. Choosing $N_0$ sufficiently large we can make $\e$ arbitrarily small, to conclude that $P_G(\phi)\ge \delta_\infty$. 
\end{proof}

Recall that the measure theoretic entropy at infinity of a transitive CMS of finite entropy $(\Sigma,\sigma)$ is defined by
\begin{equation*} 
h_\infty :=\sup_{(\mu_n)_n\to 0}\limsup_{n\to\infty}h_{\mu_n}(\sigma),
\end{equation*}
where the supremum is taken over all sequences of invariant probability measures converging on cylinders to the zero measure.  An immediate consequence of Theorem \ref{thm:main} is the following upper bound for the measure theoretic entropy at infinity of $(\Sigma,\sigma)$: 
\begin{align}  \label{eq:ineinf}
h_\infty  \le \delta_\infty 
 \end{align}
We will now prove that in fact equality holds. This is equivalent to the sharpness of the inequality in Theorem \ref{thm:main}.

\begin{proof}[Proof of Theorem~\ref{thm:vpinf}]
 As observed in \eqref{eq:ineinf}, it suffices to prove the inequality $\delta_\infty\le h_\infty$. Let $\phi: \Sigma \to \R$ be a bounded, strictly negative locally constant potential depending only on the first coordinate such that
\begin{equation*}
\lim_{n\to\infty}\sup_{x\in [n]} |\phi(x)|=0.  
\end{equation*}
By Lemma \ref{lem:testineq}, for every $t \in \R$ we have $P(t\phi)\ge \delta_\infty$. Now consider a sequence of measures $(\mu_n)_{n }$ such that 
\begin{equation*}
h_{\mu_n}(\sigma)+n\int \phi ~d\mu_n>P(n\phi)-\frac{1}{n}.
\end{equation*}
The existence of such a sequence of invariant probability measures is guaranteed by the variational principle. Then 
\begin{equation*}
h_{\mu_n}(\sigma)+n\int \phi~d\mu_n>\delta_\infty-\frac{1}{n}.
\end{equation*}
Since the potential $\phi$ is strictly negativity and bounded we conclude that the sequence  $(\mu_n)_{n}$ converges on cylinders to the zero measure. Since $h_{\mu_n}(\sigma)\ge \delta_\infty-\frac{1}{n}$, 
\begin{equation*}
\limsup_{n\to \infty} h_{\mu_n}(\sigma)\ge \delta_\infty.
\end{equation*}
In particular, $\delta_\infty\le h_\infty$. 
\end{proof}

\section{Applications} \label{sec:app}
In this section we discuss  several consequences of Theorem \ref{thm:main}. Among the consequences we obtain the upper semi-continuity of the entropy map, the entropy density of the space of ergodic measures, the stability of the measure of maximal entropy in the SPR case, existence of equilibrium states for potentials in $C_0(\Sigma)$, a relationship between the entropy at infinity and the dimension of the set of recurrent points that escape on average and a bound on the amount of mass that can escape for measures with large entropy.

\subsection{Upper semi-continuity of the entropy map}  \label{sec:usc}
Starting in the early 1970s with the work of Bowen \cite{bo1} many results describing the continuity properties of the entropy map have been obtained. More precisely, given a dynamical system $T:X \to X$, the map   $\mu \mapsto h_{\mu}(T)$ defined on the space $ \M(X,T)$  endowed with the weak$^*$ topology is called \emph{entropy map}.  In general it is not continuous \cite[p.184]{wa}. However, it was soon realised that 
that if $X$ is compact and $T$ expansive then the entropy map is upper-semi continuous \cite[Theorem 8.2]{wa}. This result has been extended to a wide range of dynamical systems exhibiting weak forms of expansion or hyperbolicity, but always assuming the compactness of $X$. Indeed, there exist examples of expansive maps $T$ defined on non-compact spaces for which the entropy map is not upper semi-continuous. We discuss some of them in this section (see Remark \ref{rem:nousc}). We recently proved in \cite[Corollary 1.2]{itv} that if $(\Sigma, \sigma)$ is a finite entropy transitive CMS then the entropy map is upper semi-continuous when restricted to ergodic measures. The method of proof used in  \cite{itv} does not seem to generalise to handle the non-ergodic case. However, the general case can be obtained directly as a corollary of  Theorem \ref{thm:main}.

\begin{theorem} \label{semicont} Let $(\Sigma,\sigma)$ be a transitive CMS of finite topological entropy and  $(\mu_n)_{n}$ a sequence of $\sigma$-invariant probability measures converging weak$^*$ to $\mu$. Then 
\begin{equation*}
\limsup_{n\to \infty} h_{\mu_n}(\sigma)\le h_\mu(\sigma).
\end{equation*}
That is,  the entropy map is upper semicontinuous.
\end{theorem}
The proof follows immediately from Theorem \ref{thm:main}, the fact that $|\mu|=1$ and Lemma \ref{restriction}.

\begin{remark}\label{rem:nousc}
We now describe the situation in the infinite entropy case.
\begin{enumerate}
\item[(a)] 
Without the finite entropy assumption, Theorem~\ref{semicont} is false, as we demonstrate here. If $(\Sigma , \sigma)$ is a a topologically transitive infinite entropy CMS then there exists a sequence $(\nu_n)_n$ and $\mu$ in  $\M(\Sigma, \sigma)$ such that $h_{\mu}(\sigma) < \infty$, and $\lim_{n \to \infty} h_{\nu_n}(\sigma)= \infty$. Let $(\mu_n)_n$ be the sequence of invariant probability measures defined by
\begin{equation*}
 \mu_n:= \left(1-\frac{1}{\sqrt{h_{\nu_n}(\sigma)}} \right)\mu+\frac{1}{\sqrt{h_{\nu_n}(\sigma)}}       \nu_n.
 \end{equation*}
Notice $\mu_n$ is well defined for large enough $n$. Then $(\mu_n)_n$ converges weak$^*$ to $\mu$ and
\begin{equation*}
h_{\mu}(\sigma)<\lim_{n \to \infty} h_{\mu_n}(\sigma)= \infty.
\end{equation*}
Therefore, the entropy map is not upper semi-continuous at any finite entropy measure.

\item[(b)]  Examples of sequences of ergodic measures with finite entropy uniformly bounded above  converging weak$^*$ to an ergodic measure (with finite entropy) in the full-shift on a countable alphabet, for which the entropy map is not upper-semi continuous can be found in \cite[p.774]{jmu} and \cite[Remark 3.11]{itv}.

\item[(c)] The entropy map is trivially upper semi-continuous at any measure of infinite entropy.

\end{enumerate}
\end{remark}

We conclude this subsection with a  consequence of Theorem \ref{thm:main} and Remark \ref{rem:nousc}. 

\begin{proposition}\label{prop:iff} Let $(\Sigma,\sigma)$ be a transitive CMS. Then $h_{top}(\sigma)$ is finite if and only if $\delta_\infty$ is finite. 
\end{proposition}

\begin{proof} We only need to prove that if $\delta_\infty$ is finite, then $h_{top}(\sigma)$ is finite; the other direction follows directly from the inequality $\delta_\infty\le h_{top}(\sigma)$. 

First assume that $(\Sigma,\sigma)$ does not satisfy the $\F-$property.  It follows directly from the definition of $\delta_\infty$ that in this situation we have $\delta_\infty=\infty$. As mentioned above there is nothing to prove in this case. 

Now assume that $(\Sigma,\sigma)$ satisfies the $\F-$property. In the proof of Theorem \ref{thm:main} we did not use the fact that the topological entropy of $(\Sigma,\sigma)$ is finite, we only used that our CMS has the $\F-$property and that $\delta_\infty$ is finite--those follow trivially under the finite entropy assumption. The $\F-$property is  crucially used in Proposition \ref{prop:atom} and Lemma \ref{lem:Kineq2}. If $\delta_\infty$ is finite, then Theorem \ref{thm:main} implies that the entropy map is upper semi-continuous, which would contradict Remark \ref{rem:nousc} if $h_{top}(\sigma)$ is infinite. We conclude that the topological entropy of $(\Sigma,\sigma)$ is finite.
\end{proof}

\subsection{Suspension flows}

Let $(\Sigma, \sigma)$ be a transitive, finite entropy  CMS and  $\tau: \Sigma \to \R^+$  a potential  bounded away from zero. Let $$Y:= \left\{ (x,t)\in \Sigma  \times \R \colon 0 \le t \le\tau(x) \right\},$$
with the points $(x,\tau(x))$ and $(\sigma(x),0)$ identified for each $x\in \Sigma $. The \emph{suspension flow} over $\Sigma$
with \emph{roof function} $\tau$ is the semi-flow $\Phi= (\phi_t)_{t \in \R_{\ge 0}}$ on $Y$ defined by
$ \phi_t(x,s)= (x,s+t)$ whenever $s+t\in[0,\tau(x)]$. Denote  by $\mathcal{M}(Y,\Phi)$ the space of flow invariant probability measures.  In this section we prove that in this continuous time, non-compact setting again the entropy map is upper semi-continuous. This generalises  \cite[Proposition 5.2]{itv} in which upper semi-continuity of the entropy map was proven for ergodic measures. Let
 \begin{equation}
\mathcal{M}_\sigma(\tau):= \left\{ \mu \in \mathcal{M}_{\sigma}: \int \tau ~d \mu < \infty \right\}.
\end{equation}
A result by Ambrose and Kakutani \cite{ak} implies that the map $M \colon \mathcal{M}_\sigma \to \mathcal{M}_\Phi$, defined by
\begin{equation*} \label{eq:R map}
M(\mu)=\frac{(\mu \times  \text{Leb})|_{Y} }{(\mu \times  \text{Leb})(Y)},
\end{equation*}
where  \text{Leb} is the one-dimensional Lebesgue measure, is a bijection.
The following result proved in \cite[Lemma 5.1]{itv}  describes the relation between weak$^*$ convergence in $\M(Y,\Phi)$ with that in $\M(\Sigma, \sigma)$.

\begin{lemma} \label{lem:weak} 
Let $(\nu_n), \nu \in \M(Y,\Phi)$ be flow invariant probability measures such that
\begin{equation*}
\nu_n=\frac{\mu_n \times Leb}{\int \tau~d \mu_n} \quad \text{ and } \quad \nu= \frac{\mu \times Leb}{\int \tau~d \mu} 		
\end{equation*}
where $(\mu_n)_n , \mu \in \M(\Sigma, \sigma)$.  If the sequence $(\nu_n)_n$ converges weak$^*$ to $\nu$ then
$(\mu_n)_n$ converges weak$^*$ to $\mu$   and  $\lim_{n \to \infty} \int \tau~d \mu_n = \int \tau~d\mu$. 		
\end{lemma}

\begin{proposition} \label{thm:susp}
Let $(\Sigma,\sigma)$ be a  transitive CMS of  finite topological entropy. Let $\tau$ be a potential bounded away from zero and $(Y,\Phi)$ the suspension flow of $(\Sigma,\sigma)$ with roof function $\tau$. Then the entropy map of  $(Y, \Phi)$ is upper semi-continuous.
\end{proposition}

The proof directly follows from Abramov's formula \cite{ab}, Lemma \ref{lem:weak}  and Theorem \ref{thm:main}. Because of the similarities between the geodesic flow and the suspension flow over a Markov shift it is reasonable to expect that, under suitable assumptions on the roof function $\tau$, the suspension flow also satisfies an entropy inequality like Theorem \ref{thm:main}. This is in fact the case and will be discussed in \cite{ve2}. The space of invariant measures for the suspension flow was already investigated and described in \cite[Section 6]{iv}.

\subsection{Entropy density of ergodic measures}  \label{sec:ed}  

The structure of the space of invariant measures for finite entropy (non-compact) CMS was studied in \cite{iv}. In this non-compact setting it is well known that the space of ergodic measures is still dense in $\M(\Sigma,\sigma)$ (see \cite[Section 6]{csc}).  A natural question is whether the approximation by ergodic measures can be arranged so that the corresponding entropies also converge. 
If this is the case we say that the set of ergodic measures is \emph{entropy dense}. More precisely,

\begin{definition} \label{def:edense}
A subset $\mathcal{L} \subset \M(\Sigma,\sigma)$ is \emph{entropy dense} if for every measure $\mu \in \M(\Sigma,\sigma)$ there exists a sequence $(\mu_n)_n$ in $\mathcal{L}$ such that
\begin{enumerate}
\item $(\mu_n)_n$ converges to $\mu$ in the weak$^*$ topology.
\item $\lim_{n\to\infty} h_{\mu_n}(\sigma)=h_\mu(\sigma)$.
\end{enumerate} 
\end{definition}
Results proving that certain classes of measures are entropy dense have been obtained for different  dynamical systems defined on compact spaces by Katok \cite{ka}, Orey \cite{or},  F\"ollmer and Orey \cite{fo}, Eizenberg, Kifer and  Weiss \cite{ekw} and by Gorodetski and Pesin \cite{gp} among others. In this section we prove, for the non-compact setting of finite  entropy CMS, that the set of ergodic measures $\E(\Sigma,\sigma)$ is entropy dense. 
 
\begin{theorem}\label{teodense}
Let  $(\Sigma, \sigma)$ be a finite entropy, transitive CMS and  $\mu \in \M(\Sigma,\sigma)$. Then there exists a sequence $(\mu_n)_{n }$ of ergodic measures such that 
$(\mu_n)_n$ converges to $\mu$ in the weak$^*$ topology and  $\lim_{n\to\infty} h_{\mu_n}(\sigma)=h_\mu(\sigma)$, i.e., $\E(\Sigma, \sigma)$ is entropy dense.  Moreover, it is possible to choose the sequence so that each $\mu_n$ has compact support.
\end{theorem}

The proof of this result directly  follows combining  Theorem \ref{semicont}, where the upper semi-continuity of the entropy map is proved, and Proposition \ref{dense}, where we proved a weak form of entropy density of the set of ergodic measures.  Note that the entropy density property of ergodic measures is an important tool in proving large deviations principles via the orbit-gluing technique (see, for example, \cite{ekw} and \cite{fo}).

\subsection{Points that escape on average} 

In this section  we relate the Hausdorff dimension of the set of  recurrent points that escape on average with the entropy at infinity of $(\Sigma,\sigma)$. Recall we have fixed an identification of the alphabet of $(\Sigma,\sigma)$ with $\N$.

\begin{definition}
Let $(\Sigma, \sigma)$ be a CMS, the set of points that \emph{escape on average} is defined by
\begin{equation*}
E:=\left\{ x \in \Sigma : \lim_{n \to \infty} \frac{1}{n} \sum_{i=0}^{n-1} 1_{[a]}(\sigma^i x)=0, \text{ for every } a \in \N 	\right\}. 
\end{equation*}
We say that $x\in\Sigma$ is a \emph{recurrent point} if there exists an increasing sequence $(n_k)_k$ such that $\lim_{k\to\infty}\sigma^{n_k}(x)=x$. The set of recurrent points is denoted by $\mathcal{R}$. 
\end{definition}

A version of the set $E$ has been considered in the context of homogeneous dynamics. Interest in that set stems from work of Dani \cite{da} in the mid 1980s who proved that singular matrices are  in  one-to-one correspondence with certain divergent orbits of one parameter diagonal groups on the space of lattices.  For example, Einsiedler and Kadyrov \cite[Corollary 1.7]{ek}  computed the Hausdorff dimension of that set in the setting of $SL_3(\Z) \backslash SL_3(\R)$. In the context of unimodular $(n+m)-$lattices an upper bound for the Hausdorff dimension of the set of points that escape on average has been obtained in \cite[Theorem 1.1]{kklm}.  More recently, for the Teichm\"uller geodesic flow,  in \cite[Theorem 1.8]{aaekmu} the authors prove an upper bound for the Hausdorff dimension of directions in which Teichm\"uller geodesics escape on average in a stratum. In all the above mentioned work, either explicitly or not, the bounds are related to the entropy at infinity of the system. Our next result establishes an analogous result for CMS. In this case the upper bound is the entropy at infinity divided by $\log 2$. This latter constant comes from the metric we consider in the space (see  \eqref{metric}) and can be thought of as the Lyapunov exponent of the system. 

 \begin{theorem}\label{thm:onave}
Let $(\Sigma, \sigma)$ be a  finite entropy transitive CMS. Then
\begin{equation*}
\dim_H(E\cap \mathcal{R}) \leq \frac{\delta_{\infty}}{\log 2}
\end{equation*}
where $\dim_H$ denotes the Hausdorff dimension with respect to the metric \eqref{metric}.
\end{theorem}

Before initiating the proof of Theorem \ref{thm:onave} let us set up some notation. Given natural numbers $a, b, q, m$ and $N$ we define $S^q_{a,b}(N,m)$ as the collection of cylinders of the form $[x_0,...,x_{L-1}]$, where $L\ge Nm$, $x_0=a, x_{L-1}=b,$ and the number of indices $i\in\{0,...,L-1\}$ such that $x_i\le q$ is exactly $N$. It will be convenient to define 
$$H_{a,b}^q(n,m):=\bigcup_{N\ge n}S_{a,b}^q(N,m).$$ 
Finally define 
$$\mathcal{L}_b:=\{x \in\Sigma: \exists (n_k)_k \text{ strictly increasing such that } \sigma^{n_k}(x)\in[b], \forall k\in\N\},$$ 
and $\mathcal{L}=\bigcup_{b\in \N} \mathcal{L}_b$. 

\begin{remark}\label{rem:enddd} Let $a, b, q$ and $m$ be natural numbers. Assume that $q\ge b$. Note that if $x\in \left(E\cap\cL_b\cap [a] \right)$, then there exists $s_0 \in \N$  such that 
$$\# \left\{i\in\{0,...,s-1\}: x_i\le q \right\}\le \frac{s}{m},$$
for every $s\ge s_0$. Moreover, there exists an  increasing sequence $(n_k)_{k}$ such that $x_{n_k}=b$. Define $\T_k(x)=\#\{i\in\{0,...,n_k-1\}: x_i\le q\}$. Since $q\ge b$ we get that $\T_k(x)\ge k$.  Observe that if $n_k\ge s_0$, then 
$$m\mathcal{T}_k(x)= m\#\{i\in\{0,...,n_k-1\}: x_i\le q\}\le n_k.$$
We conclude that $$[x_0,...,x_{n_k-1}]\in S^q_{a,b}(\mathcal{T}_k(x),m)\subset \bigcup_{p\ge k}S^q_{a,b}(p,m)=H_{a,b}^q(k,m).$$ This gives us the inclusion 
\begin{align}\label{eq:cov}
\left( E\cap \cL_b\cap [a] \right) \subset \bigcup_{C\in H_{a,b}^q(k,m)} C,\end{align}
for every $k\in \N$. 
\end{remark}

\begin{proof}[Proof of Theorem~\ref{thm:onave}] First observe  that $\left(E\cap\cR \right)\subset  \bigcup_{b\in\N} \left(E\cap\cL_b\cap [b] \right)$. In particular it suffices to prove that $\dim_H(E\cap \mathcal{L}_b\cap[a])\le \delta_\infty / \log 2$  , for every pair of natural numbers $a$ and $b$. Fix $t> \delta_\infty/\log 2$. Recall that $\delta_\infty=\inf_{m,q}\delta_\infty(m,q)$ (see equation \eqref{infinf}). Choose $m$ and $q$ large enough so that $t>\delta_\infty(q,m)/\log 2$, and that $q\ge \max\{a,b\}$. Observe that we are now in the same setup as in Remark \ref{rem:enddd}.

 In order to estimate the Hausdorff dimension of $E\cap \cL_b\cap[a]$ we will use the covering  given by \eqref{eq:cov}. Thus, it is enough to bound $\sum_{C\in H_{a,b}^q(k,m)} \diam(C)^t$.  First observe that  since  $q\ge\max\{a,b\}$,  
 a cylinder $C\in H_{a,b}^q(k,m)$ has length $\ell(C)\ge k$. Recall that $\diam(C)\le 2^{-\ell(C)}=e^{-(\log 2)\ell(C) }$. Therefore, as $k \in \N$ increases the diameter of the covering given by \eqref{eq:cov} converges to zero. Now observe that 

\begin{align*}\sum_{C\in H_{a,b}^q(k,m)} \diam(C)^t&\le  \sum_{C\in H_{a,b}^q(k,m)} e^{-t(\log 2)\ell(C)}\\
&=\sum_{l\ge k}e^{-t(\log 2) l}\#\{C: C\in H_{a,b}^q(k,m)\text{ and }\ell(C)=\ell\}\\ 
&\le\sum_{l\ge k} e^{-t(\log 2) l}z_{l-2}(m,q). \end{align*}
In the last inequality we used that $$\#\{C\in H_{a,b}^q(k,m)\text{ and }\ell(C)=l\}\le z_{l-2}(m,q).$$ Indeed, if $C\in H_{a,b}^q(k,m)\text{ and }\ell(C)=\ell$, then $C$ is a cylinder of the form $[x_0,...,x_{\ell-1}]$ where $x_0=a$, $x_{l-1}=b$, and $$\#\{i\in\{0,...,\ell-1\}: x_i\le q\}=k\le \frac{\ell}{m}.$$
Since $\max\{a,b\}\le q$ we conclude that $C$ is one of the cylinders  counted in the definition of $z_{\ell-2}(m,q)$ (see Definition \ref{def:ent_inf}). 

By the definition of $\delta_\infty(m,q)$ the series $Z(s):=\sum_{\ell=2}^\infty e^{-s\ell}z_{l-2}(m,q)$ is convergent for $s>\delta_\infty(m,q)$. In particular since $t\log 2>\delta_\infty(m,q)$ we have that $Z(t\log 2)$ is finite. This implies that the tail of $Z(t\log 2)$ converges to zero. We conclude that $\sum_{C\in H_{a,b}^q(k,m)} \diam(C)^t$ goes to zero as $k\to\infty$. This implies that $\dim_H(E\cap \cL_b\cap[a])\le t$, but $t$ was an arbitrary number larger than $\delta_\infty / \log 2$.\end{proof}

\begin{remark} It is proved in \cite[Theorem 3.1]{i} that if $(\Sigma,\sigma)$ is a transitive CMS with finite topological entropy, then $\dim_H(\cR)=h_{top}(\sigma)/\log 2$. In particular if $(\Sigma,\sigma)$ is SPR, then $\dim_H(E\cap\cR)<\dim_H(\cR)$. 
\end{remark}

 \subsection{Measures of maximal entropy}\label{sec:mme}  
 An invariant measure $\mu \in \M(\Sigma, \sigma)$ is called a \emph{measure of maximal entropy} if $h_{\mu}(\sigma)= h_{top}(\sigma)$. 
 It follows from work by Gurevich \cite{gu1,gu2} that if $h_{top}(\sigma)<\infty$ then there exists at most one measure of maximal entropy. Note that a direct consequence of the variational principle (see \cite{gu2} or 
 Theorem \ref{thm:vp}) is that there exists a sequence of invariant probability measures $(\mu_n)_n$ such that $\lim_{n \to\infty} h_{\mu_n}(\sigma)= h_{top}(\sigma)$. 
Moreover,  if the sequence has a weak$^*$ accumulation point $\mu$ then it follows from the upper semi-continuity of the entropy map, see Theorem \ref{semicont}, that $h_{\mu}(\sigma)=h_{top}(\sigma)$.  Since the space $\M(\Sigma, \sigma)$ is not compact there are cases in which the sequence $(\mu_n)_n$ does not have an accumulation point. In fact, there exist transitive finite entropy CMS that do not have measures of maximal entropy (see \cite{ru2} for a wealth of explicit examples). Our next result follows directly from Theorem \ref{thm:main} and Theorem \ref{compact}. Recall that $(\Sigma,\sigma)$ is SPR if and only if $\delta_\infty<h_{top}(\sigma)$ (see  Proposition \ref{prechar}).

\begin{theorem} \label{thm:mme}
Let $(\Sigma, \sigma)$ be a SPR CMS and $(\mu_n)_{n}$  a sequence of $\sigma$-invariant probability measures such that 
\begin{equation*}
\lim_{n\to\infty}h_{\mu_n}(\sigma)=h_{top}(\sigma).
\end{equation*}
Then the sequence $(\mu_n)_{n }$  converges in the weak$^*$ topology to the unique measure of maximal entropy. 
\end{theorem}

\begin{proof} Note that the inequality $\delta_\infty<h_{top}(\sigma)$ immediately implies that $(\Sigma,\sigma)$ has finite topological entropy (see Proposition \ref{prop:iff}). Since $\M_{\le1}(\Sigma,\sigma)$ is compact (see Theorem \ref{compact}) there exists a subsequence $(\mu_{n_k})_k$ which converges on cylinders to $\mu\in \M_{\le1}(\Sigma,\sigma)$. It follows directly from Theorem  \ref{thm:main} that
\begin{equation*}
h_{top}(\sigma)= \limsup_{k\to \infty} h_{\mu_{n_k}}(\sigma)\le |\mu|h_{\mu/|\mu|}(\sigma)+(1-|\mu|)\delta_\infty.
\end{equation*}
Recall that $\delta_{\infty} < h_{top}(\sigma)$. If $|\mu| <1$ then the right hand side of the equation is a convex combination of numbers, one of which is strictly smaller than $h_{top}(\sigma)$. Since this is not possible we have that $|\mu|=1$. In particular
\begin{equation*}
h_{top}(\sigma) \leq h_{\mu}(\sigma).
\end{equation*}
That is, $\mu$ is a measure of maximal entropy. We conclude that $(\Sigma,\sigma)$ has a measure of maximal entropy. The same argument holds for every subsequence of $(\mu_n)_n$, this implies that the entire sequence $(\mu_n)_n$ converges in the weak$^*$ topology to the unique measure of maximal entropy. 
\end{proof}

In fact Theorem \ref{thm:main} also gives a complete description of non strongly positive recurrence, as follows.  Some of these results were originally proved in \cite[Theorem 6.3]{gs} by different methods. 

\begin{theorem}\label{thm:mme2}
Let $(\Sigma, \sigma)$ be a transitive CMS of finite entropy.
\begin{enumerate}
\item Suppose $(\Sigma,\sigma)$ does not admit a measure of maximal entropy. Let $(\mu_n)_{n }$ be a sequence of $\sigma$-invariant probability measures such that $\lim_{n\to\infty}h_{\mu_n}(\sigma)=h_{top}(\sigma)$. Then  $(\mu_n)_{n }$  converges on cylinders to the zero measure and  $\delta_\infty=h_{top}(\sigma)$. 
\item Suppose that $(\Sigma,\sigma)$ is positive recurrent, but $h_{top}(\sigma)=\delta_\infty$. Let $(\mu_n)_{n}$ be a sequence of $\sigma$-invariant probability measures such that $\lim_{n\to\infty}h_{\mu_n}(\sigma)=h_{top}(\sigma)$. Then the accumulation points of  $(\mu_n)_{n}$  lie in the set $\{ \lambda \mu_{max}: \lambda \in [0,1]\}$, where $\mu_{max}$ is the measure of maximal entropy. Moreover, every measure in $\{ \lambda \mu_{max}:t\in [0,1]\}$ can be realised as such limit. 
\end{enumerate}
\end{theorem}

\begin{proof}
Note that part $(a)$ directly follows from Theorem \ref{thm:main}. Indeed, if a sequence $(\mu_n)_{n }$ with $\lim_{n\to\infty}h_{\mu_n}(\sigma)=h_{top}(\sigma)$ converges in cylinder to a measure $\mu \in \M_{\leq 1}(\Sigma, \sigma)$ different from the zero measure then $\mu/|\mu|$ would be a measure of maximal entropy. This argument also gives us the first part of $(b)$, that is, the accumulation points of $(\mu_n)_n$ lie in $\{\lambda\mu_{max}:\lambda\in[0,1]\}$.  As for the second part of $(b)$,  by  Theorem \ref{thm:vpinf} there exists a  sequence $(\mu_n)_n$ in   $\M(\Sigma,\sigma)$ with $\lim_{n \to \infty} h_{\mu_n}(\sigma)=h_{top}(\sigma)$ such that  that $(\mu_n)_n$ converges on cylinders to the zero measure. Since there exist a measure of maximal entropy $\nu$ we have that for every $\lambda \in [0,1]$ the sequence $\rho_n:=\lambda \nu +(1-\lambda)\mu_n$ converges on cylinders to $\lambda\nu$ and $\lim_{n \to \infty} h_{\rho_n}(\sigma) =h_{top}(\sigma)$.
\end{proof}

\subsection{Existence of equilibrium states}\label{sec:eqst}
In this section we will always assume that $(\Sigma,\sigma)$ is a transitive CMS with finite entropy. In Section~\ref{sec:tf} we described the thermodynamic formalism developed by Sarig in the setting of CMS and  functions (potentials) of summable variations. It turns out that the same methods can be extended and thermodynamic formalism can be developed for functions with weaker regularity assumptions (for example functions satisfying the Walters condition \cite{sabook}). However, these methods can not  be extended much further. In this section we propose an alternative definition of pressure that generalises the Gurevich pressure to the space of functions  $C_0(\Sigma)$ (see Definition \ref{C_0}). We stress that  these functions are just uniformly continuous. Making use of  Theorem \ref{thm:main} we can ensure the existence of equilibrium states. 
 
The following result is  a direct consequence of Theorem \ref{thm:main} and the continuity of the map $\mu\mapsto \int Fd\mu$, when $F\in C_0(\Sigma)$ and $\mu$ ranges in $\M_{\le1}(\Sigma,\sigma)$ endowed with the cylinder topology. 

\begin{theorem}\label{ineqC_0} Let $(\Sigma,\sigma)$ be a transitive CMS with finite entropy and $F\in C_0(\Sigma)$. Let $(\mu_n)_n$ be a sequence in $\M(\Sigma,\sigma)$ converging on cylinders to $\lambda \mu$, where $\lambda\in [0,1]$ and $\mu\in \M(\Sigma,\sigma)$. Then
$$\limsup_{n\to\infty}\left( h_{\mu_n}(\sigma)+\int Fd\mu_n\right)\le \lambda \left(h_{\mu}(\sigma)+\int Fd\mu\right)+(1-\lambda)\delta_\infty.$$
\end{theorem}

For a  continuous, bounded potential $F$ define the \emph{(variational) pressure} of $F$ by 
 \begin{equation*}
 P_{var}(F):=\sup_{\mu\in \M(\Sigma,\sigma)}\left(h_\mu(\sigma)+\int Fd\mu\right).
 \end{equation*}
A measure $\mu$ is an equilibrium state for $F$ if $P_{var}(F)=h_\mu(\sigma)+\int Fd\mu$.  Recall that since $F$ needs not to be of summable variations then the classifications of potentials (see Definition \ref{def:clas}) and the uniqueness of equilibrium states (Theorem \ref{clas}) do not necessarily hold. 

Note that if $F\in C_0(\Sigma)$, then $P_{var}(F)\ge \delta_\infty$. Indeed, let $(\mu_n)_n$ be a sequence of measures in $\M(\Sigma,\sigma)$ converging on cylinders to the zero measure and such that $\lim_{n\to\infty}h_{\mu_n}(\sigma)=\delta_\infty$. Since $F\in C_0(\Sigma)$, then $\lim_{n\to\infty}\int Fd\mu_n=0$. We conclude that $$P_{var}(F)\ge \limsup_{n\to\infty}\left(h_{\mu_n}(\sigma)+\int Fd\mu_n\right)=\delta_\infty.$$

Our next result follows directly from Theorem \ref{ineqC_0} and Theorem \ref{compact}, as Theorem \ref{thm:mme} follows from Theorem \ref{thm:main} and Theorem \ref{compact}. 

\begin{theorem}\label{thm:sta} Let $(\Sigma,\sigma)$  be a transitive CMS with finite entropy and $F\in C_0(\Sigma)$. Assume that $P_{var}(F)> \delta_\infty$. Then there exists an equilibrium state for $F$. Moreover, if $(\mu_n)_n$ is a sequence in $\M(\Sigma,\sigma)$ such that $$\lim_{n\to\infty}\left(h_{\mu_n}(\sigma)+\int Fd\mu_n\right)=P_{var}(F),$$
then every limiting measure of $(\mu_n)_n$ is an equilibrium state of $F$.  
\end{theorem} 

In Theorem \ref{thm:sta}, if  we further assume that $F$ has summable variations, then the sequence $(\mu_n)_n$ converges in the weak$^*$ topology to the unique equilibrium state of $F$. For the description of the pressure map $t\mapsto P_{var}(tF)$ we refer the reader to \cite[Theorem 5.7]{rv}.

\subsection{Entropy and escape of mass}

In this subsection we show  that 
for a SPR CMS $(\Sigma,\sigma)$ it is possible to bound the escape of mass of sequences of measures with sufficiently large entropy. In the setting of homogenous flows an analogous result was proven in  \cite[Corollary of Theorem A]{ekp}.

\begin{theorem} \label{thm:em} Let $(\Sigma,\sigma)$ be a SPR CMS. Let $(\mu_n)_n$ be a sequence in $\M(\Sigma,\sigma)$ such that $h_{\mu_n}(\sigma)\ge c$, for every $n\in\N$, and $c\in (\delta_\infty,h_{top}(\sigma))$. Then every limiting measure $\mu$ of $(\mu_n)_n$ with respect to the cylinder topology satisfies
\begin{equation*}
\mu(\Sigma)\ge \frac{c-\delta_\infty}{h_{top}(\sigma)-\delta_\infty}.
\end{equation*}
\end{theorem}

\begin{proof}
From Theorem \ref{thm:main}  we have that
\begin{eqnarray*}
c \leq \limsup_{n \to \infty} h_{\mu_n}(\sigma) \leq \mu(\Sigma) h_{\mu / |\mu|}(\sigma) + (1 - \mu(\Sigma)) \delta_{\infty} \leq 
\mu(\Sigma) (h_{top}(\sigma) - \delta_{\infty}) + \delta_{\infty}.
\end{eqnarray*}
The result then follows.
\end{proof}

\end{document}